\documentclass[11pt]{amsart}

\usepackage{amsthm, amsfonts, graphicx, pinlabel}
\usepackage[all]{xy}
\usepackage[usenames,dvipsnames,table]{xcolor}
\usepackage{hyperref}

\newcommand{\df}{\ensuremath{\partial}}

\DeclareMathOperator{\image}{im}

\def\TM+{T^*(\rr_+ \times M)}

\newcommand{\cc}{\ensuremath{\mathbb{C}}}
\newcommand{\rr}{\ensuremath{\mathbb{R}}}

\newcommand{\nn}{\ensuremath{\mathbb{N}}}
\newcommand{\ff}{\ensuremath{\mathbb{F}}}

\theoremstyle{plain}
\newtheorem{thm}{Theorem}[section]
\newtheorem{cor}[thm]{Corollary}
\newtheorem{lem}[thm]{Lemma}

\newtheorem{prop}[thm]{Proposition}

\newtheorem{ques}[thm]{Question}

\theoremstyle{definition}
\newtheorem{defn}[thm]{Definition}

\theoremstyle{remark}
\newtheorem{rem}[thm]{Remark}
\newtheorem{ex}[thm]{Example}

\numberwithin{equation}{section} 

\newcommand{\dfn}[1]{{\textbf {#1}}}
\newcommand{\leg}{\ensuremath{\Lambda}}

\def\a{\mathfrak{a}}

\begin{document}

\title[The Minimal  Length of a Lagrangian Cobordism]{The Minimal Length of a Lagrangian Cobordism between Legendrians}

\date{\today}

\author[J. Sabloff]{Joshua M. Sabloff} \address{Haverford College,
Haverford, PA 19041} \email{jsabloff@haverford.edu} \thanks{JS is
partially supported by NSF grant DMS-1406093. LT gratefully acknowledges the hospitality of
  the Institute for Advanced Study in Princeton and support at IAS  from The Fund for Mathematics during a portion of this work.}

\author[L. Traynor]{Lisa Traynor} \address{Bryn Mawr College, Bryn
Mawr, PA 19010} \email{ltraynor@brynmawr.edu} 

\begin{abstract}  
To investigate the rigidity and flexibility of Lagrangian cobordisms between Legendrian submanifolds, we investigate the minimal length of such a cobordism, which is a $1$-dimensional measurement of the non-cylindrical portion of the cobordism.  Our primary tool is a set of real-valued capacities for a Legendrian submanifold, which are derived from a filtered version of Legendrian Contact Homology.  Relationships between capacities of Legendrians at the ends of a Lagrangian cobordism yield lower bounds on the length of the cobordism.  We apply the capacities to Lagrangian cobordisms realizing vertical dilations (which may be arbitrarily short) and contractions (whose lengths are bounded below).  We also study the interaction between length and the linking of multiple cobordisms as well as the lengths of cobordisms derived from non-trivial loops of Legendrian isotopies.
\end{abstract}

\maketitle

\section{Introduction}
\label{sec:intro}

\subsection{Context and Motivation}  This paper introduces and addresses new quantitative questions about
Lagrangian cobordisms between Legendrian submanifolds.  
The cobordisms under consideration are exact, orientable Lagrangian submanifolds of the symplectization $\rr \times Y$ of a contact manifold $Y$ that in the 
complement of $[s_-, s_+] \times Y$ coincide with cylinders over Legendrians $\leg_\pm$.  
This type of Lagrangian cobordism has come under increasing scrutiny in recent years,  especially as a part of the TQFT package for a relative version of Symplectic Field Theory; see, among others, \cite{chantraine, rizell:surgery, ekholm:rsft, ekholm:lagr-cob,  ehk:leg-knot-lagr-cob}.  

Prior research into Lagrangian cobordisms investigated \emph{qualitative} questions:   fixing the Legendrians at the ends, when does such a cobordism exist?  What restrictions are there on the topology of such a cobordism?     Is the Lagrangian cobordism relation a partial order on the set of Legendrian submanifolds (up to isotopy)? 

For another viewpoint on  the rigidity and flexibility of Lagrangian cobordisms, we introduce a new \emph{quantitative} question: given two Legendrians, what is the shortest length of all Lagrangian cobordisms between them?
Here, the length of a  Lagrangian cobordism $L$, $\ell(L)$, is the infimum of the lengths of the intervals $[s_-, s_+]$ 
   such that the Lagrangian is cylindrical in the complement of $[s_-, s_+] \times Y$.
  In Section~\ref{ssec:results}, we will describe examples that show that for some pairs of Legendrian submanifolds, Lagrangian cobordisms are flexible in the sense that they can be
arbitrarily short;  for other pairs of Legendrians submanifolds, Lagrangian cobordisms are rigid in that there is a positive lower bound
to their length.

Rigidity of length brings to mind non-squeezing and more general symplectic embedding phenomena.  
That is, even when topological and volume obstructions to embedding one symplectic domain into another vanish, the embedding may still be impossible.  The original example was of embedding a ball into a cylinder as in \cite{gromov:hol}, though other classic problems involve ellipsoids and polydisks, and, more generally, embeddings of one Liouville domain into another as in \cite{hutchings:ech-cap}.  In the last case, an embedding induces a symplectic cobordism between contact boundaries of the domains in question, and the question of measuring Lagrangian cobordisms between Legendrians can be thought of as a relative version of this question.  The notion of length, however, does not seem to have an analogue in the non-relative case. In fact, an interesting feature of the length is that it arises from a $1$-dimensional measurement rather than the usual $2$-dimensional measurements that appear in non-squeezing theorems, even in the case of Lagrangian cobordisms \cite{josh-lisa:cap}.

There is one more respect in which our results are reminiscent of non-squeezing: Gromov's use of pseudo-holomorphic curves has been refined and given algebraic structure in several contexts, the most relevant for our purposes being Legendrian Contact Homology (LCH) and Symplectic Field Theory (SFT) \cite{egh}.  Using a framework reminiscent of, though independent of, Hutchings' definition of Embedded Contact Homology capacities \cite{hutchings:ech-cap}, our main tool for measuring rigidity in the length of a Lagrangian cobordism is a notion of a capacity for a Legendrian submanifold, which is  based on a filtered version of LCH, and which we will begin to describe in Section~\ref{ssec:tools}, below.

\subsection{Results}
\label{ssec:results}

To build intuition for the length of a Lagrangian cobordism, we introduce a series of examples, each of which will be explored in more detail in Sections~\ref{sec:constructions} and \ref{sec:applications}.  First, consider the Legendrian unknot $U(1) \subset \rr^3$ depicted in Figure~\ref{fig:unknot}; for this knot, the unique Reeb chord has height $1$. Let $U(v)$ denote the image of $U(1)$ under the contact diffeomorphism $(x,y,z) \mapsto (x,vy,vz)$.  Through spinning constructions, we can produce $n$-dimensional flying saucers $U^n(v)$.
 
\begin{figure}
	\labellist
	\small
	\pinlabel $1$ [r] at 45 35
	\endlabellist

\centerline{\includegraphics{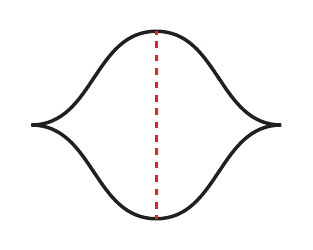}}
\caption{A front diagram of the Legendrian unknot $U(1)$, which has a single Reeb chord of height $1$.  By spinning this front around its central axis, one obtains similar ``flying saucers'', $U^n(1)$, in $\rr^{2n+1}$.}
\label{fig:unknot}
\end{figure}

\begin{thm} \label{thm:unknot-dilation}  For $v \geq 1$, there exist arbitrarily short Lagrangian cobordisms from $\leg_- = U^n(1)$ to $\leg_+ = U^n(v)$.  For $v < 1$, any such cobordism must have length at least $\ln \frac{1}{v}$, and the bound is the best possible.
\end{thm}

This theorem points to some initial intuition about the length of a cobordism:  ``expanding'' a Legendrian can be done with a cobordism of arbitrarily short length, while ``shrinking'' a Legendrian requires nontrivial length.  In fact, the lower bound in Theorem~\ref{thm:unknot-dilation} can be generalized  
to give a lower bound to the length of a Lagrangian cobordism in terms of the lengths of Reeb chords of $\leg_-$ and $\leg_+$; see Proposition~\ref{prop:chord-diff}.

Second, it is easy to show that a vertical shift of a Legendrian unknot --- or any Legendrian submanifold of $J^1M$ --- can be achieved by arbitrarily short Lagrangian cobordisms; see Corollary~\ref{cor:vert-shift}. 
 On the other hand,  simultaneous shifts of two components cannot necessarily be achieved by disjoint Lagrangian cobordisms of arbitrarily short length.
 For example, if $H(v)$ is the Hopf link of Legendrian unknots $U(1)$ with the top component shifted up by $0<v<1$ from the bottom component as in Figure~\ref{fig:hopf}, then we obtain:

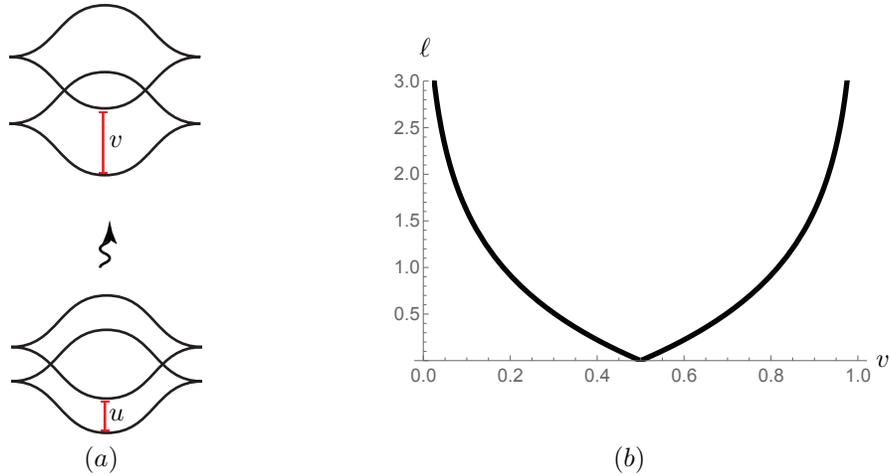
\begin{figure}
	\labellist
	\small
	\pinlabel $v$ [l] at 41 121
	\pinlabel $u$ [l] at 41 18
	\pinlabel $v$ [l] at 330 38
	\pinlabel $\ell$ [b] at 163 150
	\pinlabel $(a)$ [b] at 41 -7
	\pinlabel $(b)$ [b] at 240 -7
	\endlabellist
\centerline{\includegraphics{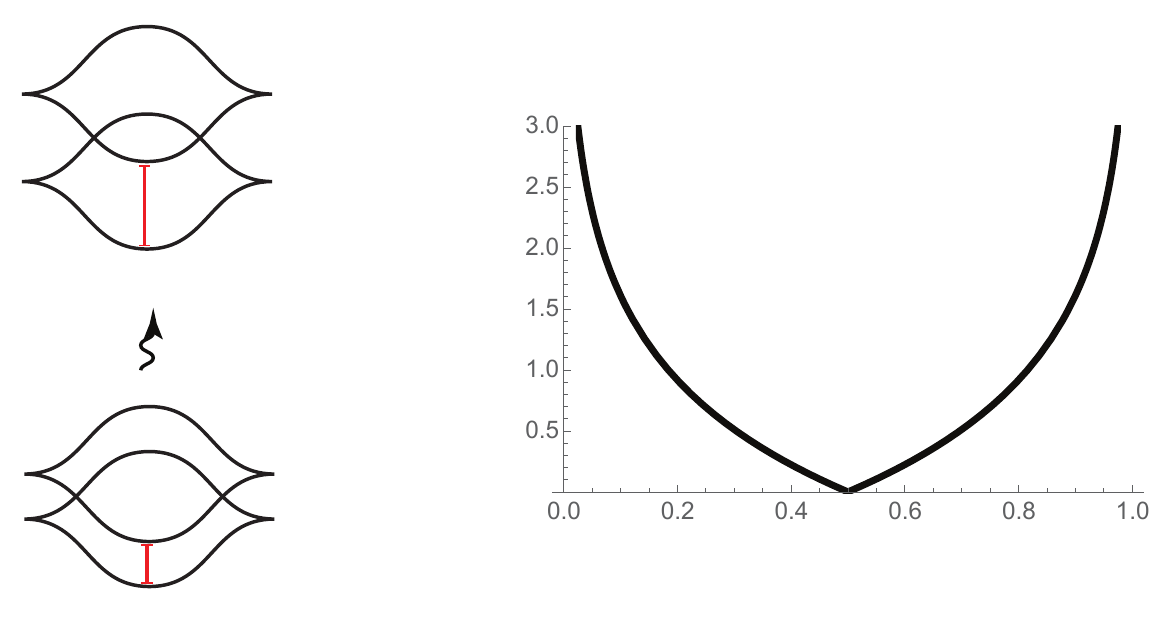}}
\caption{(a) A schematic picture of a Lagrangian cobordism between $\leg_- = H(u)$  and $\leg_+ = H(v)$, and (b) a plot of the lower bounds on the length of such a cobordism for $u = \frac{1}{2}$ and $0 < v < 1$.  }
\label{fig:hopf}
\end{figure}

\begin{thm} \label{thm:hopf}
	A Lagrangian cobordism from $\leg_- = H(u)$ to $\leg_+ = H(v)$ composed of two disjoint Lagrangian cylinders that join the upper (resp. lower) component of $H(u)$ to the upper (resp. lower) component of $H(v)$ has length at least 
	\begin{itemize}
	\item $\ln \frac{1-u}{1-v}$ if $u\leq v$ and
	\item $\ln \frac{u}{v}$ if $u \geq v$,
	\end{itemize} 
	and the bounds are the best possible.
\end{thm}
\noindent
See Figure~\ref{fig:hopf}(b) for a visualization of these bounds for $u = \frac{1}{2}$.  In fact, this statement holds for $n$-dimensional Hopf links;
see Theorem~\ref{thm:n-dim-hopf}.

Theorem~\ref{thm:hopf} hints at a boundary-value problem for Lagrangian cobordisms. Generalizing the notation for the Hopf link, let $H(v_1, \ldots, v_k)$ denote the $(k+1)$-copy of $U(1)$, with components numbered from $0$ to $k$, with the $i^{th}$ component shifted up by a height  $v_i$ from the bottom, and with the shifts satisfying $0 < v_1 < \cdots < v_k < 1$.  We call $H\left(\frac{1}{k+1}, \ldots, \frac{k}{k+1}\right)$ the \dfn{evenly shifted $(k+1)$-copy} of $U(1)$.

\begin{ques} \label{ques:packing}
	Let $H$ denote the evenly shifted $(k+1)$-copy of $U(1)$. For which vectors $(v_1, \ldots, v_k)$ is there a Lagrangian cobordism 
	from $\leg_- = H$ to $\leg_+ = H(v_1, \ldots, v_k)$  of length at most $1$?
	\end{ques}

\begin{figure}
	\labellist
	\small
	\pinlabel $1$ [r] at 90 7
	\tiny 
	\pinlabel $\frac{1}{4}$ [l] at 98 31
	\pinlabel $\frac{1}{4}$ [l] at 107 41
	\pinlabel $\frac{1}{4}$ [l] at 116 50
	\pinlabel $v_1$ [l] at 242 100
	\pinlabel $v_2$ [l] at 251 107
	\pinlabel $v_3$ [l] at 260 121
	\endlabellist
	\centerline{\includegraphics{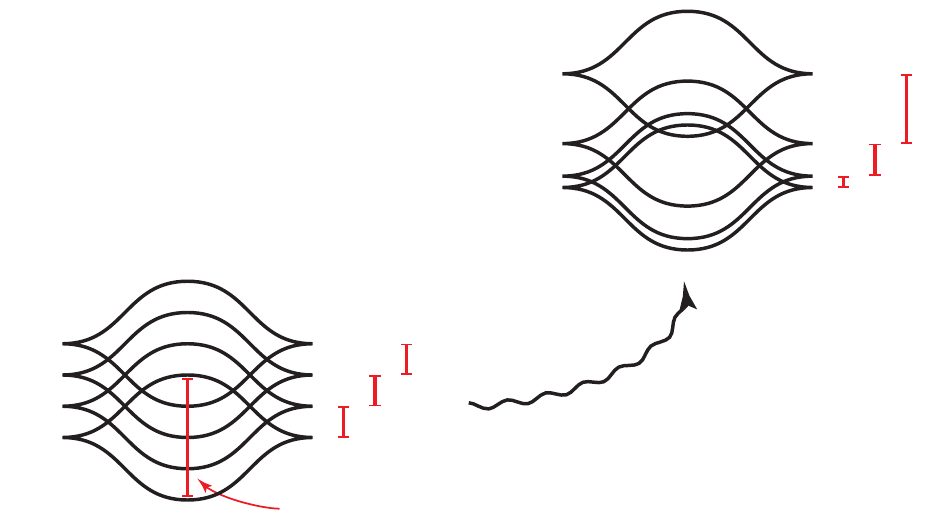}}
	\caption{The setup for the packing problem introduced in Question~\ref{ques:packing}.}
	\label{fig:packing}
\end{figure}
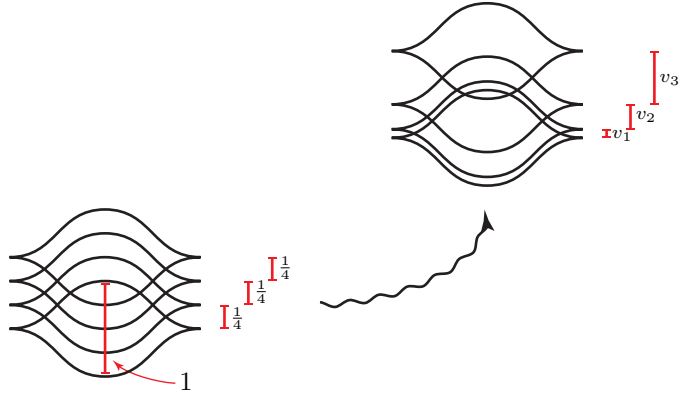
\noindent
See Figure~\ref{fig:packing}.  We will answer this question in Theorem~\ref{thm:packing} by showing that a Lagrangian cobordism exists for 
 $(v_1, \dots, v_k)$ if and only if the following system of linear inequalities are satisfied:
 \begin{equation} \label{eq:pack-ineq}
 \frac{i-j}{(k+1)e} \leq v_i-v_j \leq 1- \frac{(k+1) - (i-j)}{(k+1)e},
 \end{equation}
	where $i\in \{1, \ldots, k\}$, $j \in \{0, \ldots, k-1\}$, $i>j$, and $v_0=0$.
In the case of the $3$-copy, these inequalities are visualized in Figure~\ref{fig:packing-ineq}.
\begin{figure}
		\labellist
		\small
		\pinlabel $v_1$ [l] at 163 17
		\pinlabel $v_2$ [b] at 17 163
		\pinlabel $\frac{2}{3e}$ [t] at 54 15
		\pinlabel $\frac{2}{3e}$ [r] at 15 54
		\pinlabel $\frac{1}{3e}$ [t] at 35 15
		\pinlabel $\frac{1}{3e}$ [r] at 15 35
		\pinlabel $\frac{3e-1}{3e}$ [t] at 145 15
		\pinlabel $\frac{3e-1}{3e}$ [r] at 15 145
		\pinlabel $\frac{3e-2}{3e}$ [t] at 126 15
		\pinlabel $\frac{3e-2}{3e}$ [r] at 15 126
		\endlabellist
	\centerline{\includegraphics{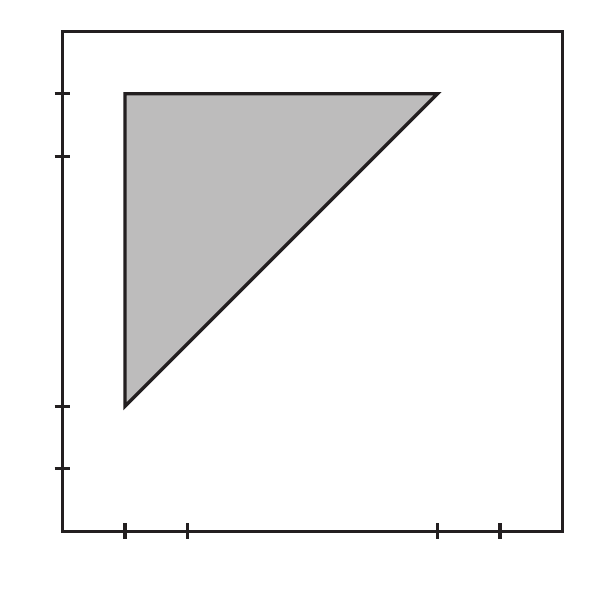}}
	\caption{The shaded region depicts the realizable shifts $v_1$, $v_2$ of the top two components of a $3$-copy as specified by Inequalities~(\ref{eq:pack-ineq}).}
	\label{fig:packing-ineq}
\end{figure}
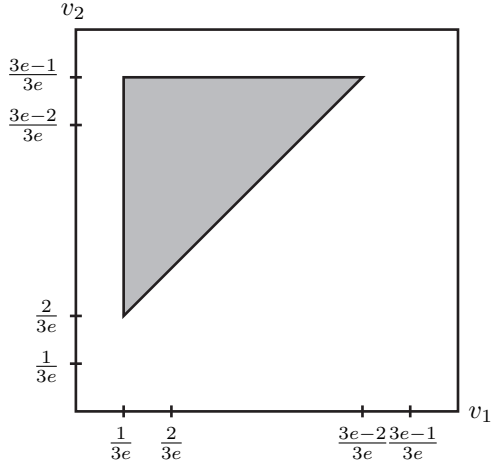

Finally, it is well known (see \cite{chantraine} and also \cite{bst:construct, ehk:leg-knot-lagr-cob, eg:finite-dim, golovko:tb}) that a Legendrian isotopy induces a Lagrangian cobordism. We consider Lagrangian cobordisms derived from loops of Legendrian submanifolds, i.e.\ isotopies that begin and end at the same Legendrian submanifold.  Unlike in the previous examples,  we
will now restrict attention to a given isotopy class of Lagrangian cobordisms: for a Lagrangian cobordism $L_\gamma$ derived from a loop of Legendrian
submanifolds $\gamma$,
we denote the set of Lagrangian cobordisms that are Lagrangian isotopic to $L_\gamma$ through compactly supported isotopies  by $[L_\gamma]$. 
Clearly if the loop $\gamma$ is trivial, then there will be Lagrangian cobordisms of arbitrarily short length in $[L_\gamma]$.  If
the loop is non-trivial, the corresponding isotopy class of Lagrangian cobordisms may or may not contain short representatives:

\begin{thm} \label{thm:loops} 
For any $b > 0$, there exists a non-trivial loop $\gamma^b$ of Legendrian submanifolds with induced Lagrangian cobordism $L_{\gamma^b}$ such that the length of any Lagrangian cobordism in $\left[L_{\gamma^b}\right]$ is bounded below by $b$.
 On the other hand, there exists a non-trivial loop of Legendrian submanifolds $\gamma'$ with induced Lagrangian cobordism $L_{\gamma'}$ such that $[L_{\gamma'}]$ contains arbitrarily short cobordisms.   

\end{thm}
\noindent
The non-trivial loop $\gamma$ that leads to long cobordisms comes from a loop of Legendrian trefoil knots constructed by K\'alm\'an \cite{kalman:mono1}. The non-trivial loop $\gamma'$ that leads to short cobordisms comes from a loop of Legendrian spheres in $\rr^{2n+1}$ constructed by the first author and Sullivan \cite{ss:pi-k}.

\subsection{Tools}
\label{ssec:tools}

The results above require two types of tools:  those that allow us to construct Lagrangian cobordisms and examine their lengths, and those that allow us to obstruct the existence of a short Lagrangian cobordism. To prove upper bounds on the length of a Lagrangian cobordism, we re-examine existing constructions of Lagrangian cobordisms  using Legendrian isotopies as in \cite{ehk:leg-knot-lagr-cob}, slightly refined with the notion of length in mind.  
To obtain the lower bounds on length, we develop a framework for capacities of Legendrian submanifolds, and show how those capacities are related under Lagrangian cobordism. 

Our capacities are derived from  a filtered version of Legendrian Contact Cohomology, linearized by an augmentation $\varepsilon$; these capacities are, in a sense, monotonic under Lagrangian cobordism.  Note that this framework --- though not the actual construction --- is similar to that used in  for slices of ``flat-at-infinity'' Lagrangians \cite {josh-lisa:cap} and to Hutchings ECH capacities \cite{hutchings:ech-cap}.  More specifically, for each linearized Legendrian Contact Cohomology  class $\theta \in LCH^*(\leg, \varepsilon)$, we  define  a quantity $c(\leg, \varepsilon, \theta) \in (0,+\infty]$ 
  that, essentially, measures the Reeb height of the class $\theta$.  Ekholm showed that a Lagrangian cobordism $L$ from $\leg_-$ to $\leg_+$ with an augmentation $\varepsilon_-$ for $\leg_-$ induces an augmentation $\varepsilon_+$ of $\leg_+$ and a map
\[ \Psi^{L, \varepsilon_-}_1: LCH^*(\leg_-,\varepsilon_-) \to LCH^*(\leg_+,\varepsilon_+).\]
We show that this map is actually a filtered map, and prove the following inequality:

\begin{thm} \label{thm:mono}
	Given a Lagrangian cobordism $L$ from  $\leg_-$ to $\leg_+$ that is cylindrical outside $[s_-,s_+]$,  an augmentation $\varepsilon_-$ of $\leg_-$, and the augmentation $\varepsilon_+$ of $\leg_+$ induced by $L$,   the following inequality holds for any $\theta \in LCH^*(\leg_-, \varepsilon_-)$: 
\begin{equation}
e^{s_-} c(\leg_-, \varepsilon_-,\theta) \leq e^{s_+}c(\leg_+, \varepsilon_+, \Psi_1^{L,\varepsilon_-}(\theta)).
\end{equation}
		Thus, for any augmentation $\varepsilon_-$ of $\leg_-$ and any $\theta \in LCH^*(\leg_-, \varepsilon_-)$ not in $\ker \Psi_1^{L,\varepsilon_-}$, we may bound the length of $L$ by 
\begin{equation} \label{eq:length-mono}
		\ell(L) \geq \ln c(\leg_-, \varepsilon_-,\theta) - \ln c (\leg_+, \varepsilon_+, \Psi_1^{L,\varepsilon_-}(\theta)).
\end{equation}
\end{thm}

As hinted by the subscript $1$ in the notation for the cobordism map, there is actually a full $A_\infty$ map between the Legendrian Contact Cohomology $A_\infty$ algebras of $\leg_\pm$.  A more general version of Theorem~\ref{thm:mono} holds for the components $\Psi^{L, \varepsilon_-}_k$ of the $A_\infty$ map; see Theorem~\ref{thm:a-infty-mono}.

\begin{rem}
	The capacity framework may also be developed using Generating Family Homology using the tools in \cite{josh-lisa:obstr}, but we chose to focus on the Legendrian Contact Homology tools in this paper.  
\end{rem}

The remainder of the paper is structured as follows:  in Section~\ref{sec:lagr-cob}, we set down  precise definitions, and in Section~\ref{sec:constructions}, we describe several constructions of Lagrangian cobordisms that will prove useful in later sections.  In Section~\ref{sec:capacity}, we establish the framework for capacities from Legendrian Contact Homology, culminating in a proof of Theorem~\ref{thm:mono} and its generalization to $A_\infty$ maps.  Applications of the capacities, including proofs of Theorems~\ref{thm:unknot-dilation}, \ref{thm:hopf}, and \ref{thm:loops}, are then given in Section~\ref{sec:applications}.  

\subsection*{Acknowledgements}  We thank Baptiste Chantraine and Richard Hind for several stimulating conversations about the foundations of capacities and the example in Section~\ref{ssec:hopf}, respectively.   We also thank the referee for insightful suggestions.

\section{Legendrian Submanifolds and Lagrangian Cobordisms}
\label{sec:lagr-cob}

In this section, we set notation  and specify the precise definition of a Lagrangian cobordism between two Legendrian submanifolds. Throughout the paper, we assume familiarity with basic ideas in contact and symplectic topology, especially with regards to Legendrian and (exact) Lagrangian submanifolds.  See, for example, Etnyre  \cite{etnyre:knot-intro} for background on Legendrian submanifolds and Audin, Lalonde, and Polterovich \cite{audin-lalonde-polterovich} on Lagrangian submanifolds. 

\subsection{Notation for Legendrian Submanifolds}
\label{ssec:legendrian-background}

We will work primarily in the $1$-jet space of a smooth manifold $M^m$ where $M$ is either compact or equal to $\rr^m$.  Topologically, we have $J^1M = T^*M \times \rr$ with local coordinates $(\mathbf{x}, \mathbf{y}, z)$.  A $1$-jet space is naturally a contact manifold with contact form $\alpha = dz - \mathbf{y} \cdot d\mathbf{x}$. 

A \dfn{Legendrian submanifold} $\leg \subset J^1M^m$ is an $m$-dimensional submanifold that is everywhere tangent to the contact structure $\xi = \ker \alpha$.  A Legendrian submanifold has two important projections:  the \dfn{front projection} to $M \times \rr$ and the \dfn{Lagrangian projection} to $T^*M$.  Note that the image of a Legendrian submanifold under the Lagrangian projection is an exact and possibly immersed Lagrangian submanifold.  

A \dfn{Reeb chord} of a Legendrian submanifold $\leg$ is an integral curve of the Reeb vector field, which is $\frac{\partial}{\partial z}$ in our setting,
whose endpoints both lie on $\leg$.  In particular, a Reeb chord will have 
constant $T^*M$ coordinates and will go from a lesser to a greater $z$ value.  Denote the collection of Reeb chords of $\leg$ by $\mathcal{R}_\leg$. The  \dfn{height of a Reeb chord} $b$  is:
\begin{equation} \label{eqn:height} h(b) = \int_b \alpha  = \int_b dz > 0.\end{equation}
Reeb chords are clearly in bijective correspondence with double points of the Lagrangian projection of $\leg$.  A Legendrian submanifold is \dfn{chord generic} if the corresponding double points of the Lagrangian projection of $\leg$ are transverse.

\subsection{Exact Cobordisms, Primitives, and Lengths}
\label{ssec:length}

We next discuss the formal definition of the Lagrangian cobordisms we consider in this paper.  Let $Y$ be an odd-dimensional manifold with contact form $\alpha$. We work  in the symplectization $(\rr \times Y, d(e^s\alpha))$ of $Y$, where $s$ denotes the coordinate on the $\rr$ factor.  

\begin{defn} \label{defn:cobordism}
	An \dfn{(exact, {orientable}, cylindrical-at-infinity) Lagrangian cobordism} is an exact, orientable Lagrangian submanifold $L \subset (\rr \times Y, d(e^s\alpha))$ such that there exist Legendrian submanifolds $\leg_\pm$ of $Y$ and real numbers $s_- \leq s_+$ satisfying:
	\begin{enumerate}
	\item $L \cap \left( (-\infty, s_-] \times Y \right) = (-\infty, s_-] \times \leg_-$,
	\item $L \cap \left([s_+,\infty) \times Y \right) = [s_+,\infty) \times \leg_+$, and
	\item The primitive of $e^s \alpha$ along $L$ is constant for $s < s_-$ and for $s > s_+$.\footnote{See \cite{chantraine:disconnected-ends} for an explanation of the third condition.  It is easy to check that the third condition is automatically satisfied if the Legendrians at the ends are connected. In general,  the primitive of $e^s \alpha$ along $L$ is constant on each component of $\leg_\pm$.}
	\end{enumerate}
	We say that $L$ is \dfn{cylindrical} outside of $[s_-,s_+]$.   We will say such a Lagrangian is a Lagrangian cobordism
	from $\leg_-$ to $\leg_+$.   
\end{defn}

For succinctness, we will henceforth drop the qualifiers ``exact, orientable, cylindrical-at-infinity'' when referring to Lagrangian cobordisms, though they are still understood to hold.  

The primary object of study in this paper is:

\begin{defn}
	The \dfn{length} of a Lagrangian cobordism $L$ is 
	\[\ell(L) = \inf \{s_+-s_-\;:\; L \text{ is cylindrical outside of } [s_-,s_+]\}.\]
\end{defn}

\section{Constructions of Lagrangian Cobordisms}
\label{sec:constructions}

In recent years, several methods for constructing Lagrangian cobordisms have been developed, yielding constructions based on Legendrian isotopy \cite{bst:construct, chantraine, ehk:leg-knot-lagr-cob, eg:finite-dim, golovko:tb}, spinning \cite{bst:construct, golovko:higher-spin}, and Lagrangian handle attachment \cite{bst:construct, rizell:surgery, ehk:leg-knot-lagr-cob}. 
In this section,  we will review a  method  to construct a Lagrangian
cobordism from a Legendrian isotopy and will then apply this construction to  give upper bounds on lengths of Lagrangian cobordisms induced by Legendrian isotopies.

We begin by reviewing the construction of a Lagrangian cobordism induced by a Legendrian isotopy from \cite[\S6]{ehk:leg-knot-lagr-cob}.  Note that the original construction in \cite{ehk:leg-knot-lagr-cob} was performed for Legendrian links in the standard contact $\rr^3$, but the proof goes through almost word-for-word in the more general setting.

Let $\leg_s$, with $s \in \rr$, be a smooth $1$-parameter family of closed, but not necessarily connected, Legendrian submanifolds of $J^1M$, where $\leg_s = \leg_-$ for $s \leq s_-$ and $\leg_s = \leg_+$ for $s \geq s_+$.  Parametrize the Legendrians in this isotopy by $\lambda_s: \Sigma \to J^1M$, where we write $\lambda_s(t) = (\mathbf{x}(s,t),\mathbf{y}(s,t),z(s,t))$.  While the trace of the isotopy in $\rr \times J^1M$ is not necessarily Lagrangian, we may use the function
\begin{equation} \label{eqn:eta}
	\eta(s, t) = \alpha ( \partial_s \lambda_s(t) )
\end{equation}
to perturb the trace into a (potentially immersed) exact Lagrangian cobordism.

\begin{lem}[\cite{ehk:leg-knot-lagr-cob}] \label{lem:isotopy-cobord} 
The map $\Gamma: \rr \times \Sigma \to \rr \times J^1M$ defined by 
\begin{equation} \label{eqn:immersion}
	\Gamma(s,t) = (s, \mathbf{x}(s,t), \mathbf{y}(s,t), z(s,t) + \eta(s,t))
\end{equation}
is an exact Lagrangian immersion.  If $\eta(s,  t)$ is
sufficiently small, then the image of $\Gamma$ is an exact Lagrangian cobordism from $\leg_-$ to  $\leg_+$. 
\end{lem}

The proof is a direct calculation that makes essential use of the fact that $\lambda_s(t)$ is Legendrian for each $s$, yielding a primitive $e^s \eta(s,t)$ for the pullback $\Gamma^*(e^s\alpha)$.  Since $\partial_s \lambda_s(t)$ has compact support, the primitive vanishes at  both ends of the cobordism.
 
 Observe that by ``spreading out'' the Legendrian isotopy, we can  
 guarantee that $\eta(s,t)$ is sufficiently small, and hence isotopic Legendrians can always
 be connected by a long Lagrangian cobordism.  We may, however, analyze when the immersion $\Gamma$ is an embedding more precisely:
 
\begin{lem}  \label{lem:reeb-cobord}
	The exact Lagrangian immersion $\Gamma$  from (\ref{eqn:immersion}) is an embedding if for all $s \in \rr$, no Reeb chord of $\leg_s$ (from $\lambda_s(t)$ to $\lambda_s(t')$) has height
$\eta(s,t) - \eta(s,t')$.
\end{lem}

\begin{proof}   A double point of the image of $\Gamma$ comes from points $(s,t)$ and $(s',t')$ that satisfy $s=s'$, $\mathbf{x}(s,t) = \mathbf{x}(s,t')$, and $\mathbf{y}(s,t) = \mathbf{y}(s,t')$ --- in particular, $\lambda_s(t)$ and $\lambda_s(t')$ must be endpoints of a Reeb chord of $\leg_s$.  Equality of the last coordinate of the immersion $\Gamma$ shows that $z(s,t) + \eta(s,t) = z(s,t') + \eta(s,t')$. Assume without loss of generality that $z(s,t) < z(s,t')$.
We then see that double points occur when  the height of that Reeb chord from $\lambda_s(t)$ to $\lambda_s(t')$ equals $\eta(s,t) - \eta(s,t')$.  
\end{proof}

From this, we see that two Legendrians related by a strict contact isotopy can be connected by an arbitrarily short Lagrangian cobordism:

\begin{prop} \label{prop:form-preserving-diffeo}  Let $\leg \subset J^1M$ be a Legendrian submanifold.  If $\Phi: J^1M \times \rr \to J^1M$ is a contact isotopy such that $\Phi_s = id$ for $s\leq 0$, $\Phi_s = \Phi_1$ for $s\geq 1$, and each $\Phi_s$ preserves the contact form $\alpha$, then there exist arbitrarily short Lagrangian cobordisms from $\leg_- = \leg$ to $\leg_+ = \Phi_1(\leg)$.
\end{prop}

\begin{proof} 
Let $X_s$ be the contact vector field generating $\Phi_s$ and note that $\eta$ is the restriction of $\alpha(X_s)$ to the trace of $\leg$ under $\Phi_s$.  Since $\Phi_s$ preserves the contact form, we may compute:
\begin{equation} \label{eq:lie-alpha}
0 = \mathcal{L}_{X_s} \alpha = d(\alpha(X_s)) + \iota_{X_s} d\alpha.
\end{equation}
Plugging the Reeb field $\partial_z$ into both sides of Equation (\ref{eq:lie-alpha}) and using that $d\alpha(\partial_z, \cdot) = 0$ tells us that $0 = d(\alpha(X_s))(\partial_z) = \partial_z(\alpha(X_s))$. 
  That is, $\eta$ is constant along Reeb chords,  hence $\eta(s,t) - \eta(s,t') = 0$ at the endpoints of any Reeb chord.  Since the height of a Reeb chord is never $0$, Lemma~\ref{lem:reeb-cobord} implies that the Lagrangian induced by the Legendrian isotopy $\leg_s$ is always embedded.  
\end{proof}

As an immediate application, we see that any two Legendrians that are related by a ``vertical shift" can be connected by an arbitrarily short Lagrangian cobordism.

 \begin{cor} \label{cor:vert-shift}
	Given a Legendrian $\leg \subset J^1M$, for any $\nu \in \rr$, let $\leg^\nu$  denote the Legendrian that is a vertical translation of $\leg$ 
	by $\nu$:
	 if $\leg$ is parameterized by $\lambda(t) = (x(t),y(t),z(t))$, $\leg^\nu$ is parameterized by $\lambda(t) = (x(t),y(t),z(t)+ \nu)$.
	Then, for any $\nu \in \rr$,  
	there exist arbitrarily short cobordisms from $\leg_- = \leg$ to $\leg_+ = \leg^{\nu}$.
\end{cor}

\begin{rem}\label{rmk:vert-distort}   For later purposes, it will be useful to more carefully examine the Lagrangian generated in the proof of Corollary~\ref{cor:vert-shift}.
By construction, the Lagrangian cobordism is the image of 
$$\Gamma(s,t) = (s, \mathbf{x}(t), \mathbf{y}(t), z(t) + \rho(s) + \rho'(s)).$$
 For any fixed $\nu \neq 0$, as the cobordism gets shorter, $\|\rho'\|_\infty$ must get larger.  Thus, the 
Lagrangian will be quite different than the trace of the isotopy.   
\end{rem}

In addition to being able to realize vertical displacements by arbitrarily short cobordisms, 
Legendrians related by a horizontal displacement or other  
changes that result from
transformations of the base coordinates can be connected by
 arbitrarily short cobordisms.

\begin{cor} \label{cor:base-diffeo}   Let $\leg \subset J^1M$ be parametrized by $\lambda(t) = (x(t),y(t),z(t))$ and suppose
$\Phi$ is a diffeomorphism of $M$ that is isotopic to the identity.  Let $\leg_\Phi$ denote the corresponding Legendrian that is parameterized by  \[\lambda_\Phi(t) = \left(\Phi(x(t)), \, [D\Phi(x(t))^{-1}]^T y(t), \,z(t)\right).\] 
Then there are  arbitrarily short Lagrangian cobordisms from $\leg_-= \leg$ to
$\leg_+ = \leg_\Phi$.
\end{cor}

It is also natural to consider Legendrians whose fronts differ by vertical expansions and contractions.  Explicit constructions show that
vertical expansions can be achieved by arbitrarily short cobordisms, while vertical contractions require some length using these 
constructions:

\begin{prop} \label{prop:scaling} Let $\leg \subset J^1M$ be a Legendrian submanifold.  Let $\leg^\sigma$ denote the image of $\leg$ under the $yz$ scaling contact diffeomorphism $(x,y,z) \mapsto (x,\sigma y, \sigma z)$.
\begin{enumerate}
\item If $\sigma > 1$, then there are arbitrarily short Lagrangian cobordisms from $\leg_- = \leg$ to $\leg_+ = \leg^\sigma$.
\item  If $\sigma < 1$, then there exists a Lagrangian cobordism from  $\leg_- = \leg$ to $\leg_+ = \leg^\sigma$ of length arbitrarily close to $\ln \frac{1}{\sigma}$.
\end{enumerate}
\end{prop}

\begin{proof}  Parametrize the Legendrian $\leg$ by $\lambda(t) = (x(t),y(t),z(t))$, and suppose $\leg$ has $k$ Reeb chords of heights
$0 < h_1 \leq h_2 \leq \dots \leq h_k$. 
Let $\rho: \rr \to \rr$ be a smooth function that is equal to $1$ for $s \leq 0$ and is equal to   $\sigma$ for $s \geq A$.  Consider the Legendrian isotopy 
	\begin{equation*}
		\lambda_s(t) = (x(t), \rho(s)y(t), \rho(s)z(t)).
	\end{equation*}
It follows that $\eta(s,t) = \rho'(s)z(t)$.    The Legendrian $\leg_s$ given as the image of $\lambda_s$ will again have
$k$ Reeb chords, now of heights 
$0 < \rho(s) h_1 \leq \rho(s) h_2 \leq \dots \leq \rho(s) h_k$.  
For every Reeb chord of $\leg_s$ of height $\rho(s)h_i$, there will be a pair of points $t, t'$ with 
$\eta(s,t) - \eta(s, t') = \rho'(s) (-h_i)$.  Thus 
we see that the embedding condition in Lemma~\ref{lem:reeb-cobord} is guaranteed when
 $\rho(s)  \neq -\rho'(s)$, for all $s$; equivalently, the condition is
 \begin{equation} \label{eqn:scaling}
 \frac{d}{ds} (e^s \rho(s)) \neq 0.
 \end{equation}
Since $\rho(s) = 1$ when $s \leq 0$, this equation is satisfied if and only if $e^s \rho(s)$ is a strictly increasing function.
 
It is possible to choose $\rho(s)$ that satisfies the boundary conditions $\rho(0) = 1$ and $\rho(A) = \sigma$ in addition to Equation (\ref{eqn:scaling}) whenever 
\begin{equation} \label{eq:dilation}
1 = e^s \rho(s) |_{s = 0} < e^s \rho(s) |_{s = A} = e^A \sigma.
\end{equation}
For an expansion, i.e. when $\sigma > 1$, Equation~(\ref{eq:dilation}) is satisfied for any $A>0$; statement (1) follows. For a contraction, i.e. when $\sigma < 1$,  we can 
 construct $\rho(s)$ so that $e^s \rho(s)$ is strictly increasing as long as 
 $\ln \frac{1}{\sigma} < A$; statement (2) follows.
\end{proof}

Are the cobordisms constructed in the proof of Proposition~\ref{prop:scaling}(2) the shortest possible?  We will return to this question in Section~\ref{sec:applications} after developing tools to bound the length of a cobordism.

\section{Capacities from Legendrian Contact Homology}
\label{sec:capacity}

In this section, we will define a set of numbers from $(0, \infty]$ that can be associated to a Legendrian submanifold that is equipped with an augmentation $\varepsilon$ of its DGA.    These capacities, defined using a filtered version of Linearized Legendrian Contact Cohomology,  will \emph{not} be Legendrian invariants;  however, as will be shown in Section~\ref{sec:cobord-maps},  these capacities can  give lower bounds to the length of a Lagrangian cobordism.

\subsection{Legendrian Contact Homology}
\label{ssec:lch-review}

The Legendrian Contact Homology (LCH) differential graded algebra is motivated by the infinite-dimensional Morse-Floer theory of the action functional on the relative path space of a Legendrian submanifold $\leg$.  The analytic framework for Legendrian Contact Homology was developed by Eliashberg \cite{yasha:icm, egh} and was first made rigorous for $1$-dimensional Legendrians using combinatorial methods by Chekanov \cite{chv}.  The theory was established for  higher dimensional Legendrian submanifolds in $1$-jet spaces by Ekholm, Etnyre, and Sullivan   \cite{ees:high-d-analysis, ees:high-d-geometry,  ees:pxr}.   Our description below owes more to \cite{egh} and to Dimitroglou Rizell's  translation between the two perspectives \cite{rizell:lifting}.

We begin our brief synopsis with a chord-generic Legendrian submanifold $\leg$ of $J^1M$ with its standard contact structure $\xi = \ker \alpha$,
where $M$ is either  a compact manifold or   $\rr^{n}$ for $n \geq 1$.  
 As Legendrian Contact Homology takes the form of a differential graded algebra, we first define $A_\leg$ to be the vector space generated by the set of Reeb chords  $\mathcal{R}_\leg$ over the field $\ff_2$ of two elements.  We then define $\mathcal A_\leg$ to be the unital tensor algebra $TA_\leg = \bigoplus_{i = 0}^\infty A_\leg^{\otimes i}$.  The generators of $A_\leg$ are graded by a Conley-Zehnder index, with the grading extended to all of $\mathcal{A}_\leg$ by letting the grading of a word be
the sum of the gradings of its constituent generators.  The gradings are well-defined up to the Maslov number of the Lagrangian projection of $\leg$. If $\leg$ has $n$ components, then the gradings of Reeb chords between different components are defined up to a shift that is constant for all Reeb chords between the same two components.

The differential $\partial_\leg: \mathcal{A}_\leg \to \mathcal{A}_\leg$ comes from a count of rigid moduli spaces of $J$-holomorphic disks.  More specifically, first choose a sufficiently generic, compatible, and $\rr$-invariant almost complex structure $J$  on the symplectization $(\rr \times J^1M, d(e^s \alpha))$ satisfying
$J(\xi) = \xi$ and $J(\partial_s) = \partial_z$.  Let $D_k$ denote the closed unit disk in $\cc$ with $k+1$ punctures $\{z_0, \ldots, z_k\}$ on its boundary. Given a Reeb chord $a \in \mathcal R_\leg$ and a monomial $\mathbf{b} = b_1 \cdots b_k$ in $\mathcal{A}_\leg$, consider the set of $J$-holomorphic maps $u: D_k \to \rr \times J^1M$ that satisfy:
\begin{enumerate} 
\item The boundary of $D_k$ maps to $\rr \times \leg$;
\item Near the puncture $z_0$, the $\rr$ coordinate of $u$ approaches $+\infty$ and the $J^1M$ coordinate is asymptotic to $\rr \times a$; and
\item Near the puncture $z_l$, for $1 \leq l \leq k$, the $\rr$ coordinate of $u$ approaches $-\infty$ and the $J^1M$ coordinate is asymptotic to $\rr \times b_l$.
\end{enumerate}
See Figure~\ref{fig:j-disk}(a). The moduli space $\mathcal{M}^J_{\rr \times \leg}(a;\mathbf{b})$ is the set of such maps up to reparametrization. Generically, the moduli space is a manifold that is invariant under the $\rr$-action induced by translation in the symplectization \cite{ees:high-d-analysis,ees:pxr}. 
\begin{figure}
\labellist
\tiny
\pinlabel $M$ [l] at 240 240
\pinlabel $M$ [r] at 425 240
\pinlabel $\rr$ at 217 130
\pinlabel $\rr$ at 442 130
\pinlabel $\leg \times \rr$ [l] at 175 180
\pinlabel $L$ [l] at 585 130
\pinlabel $\leg_+$ [l] at 593 240
\pinlabel $\leg_-$ [l] at 567 20
\pinlabel $a$ at 120 240
\pinlabel $b_1$ at 105 20
\pinlabel $b_2$ at 137 27 
\pinlabel $a$ at 540 240
\pinlabel $b_1$ at 525 20
\pinlabel $b_2$ at 552 27 
\pinlabel $z_0$ [b] at 330 75
\pinlabel $z_1$ [r] at 310 10
\pinlabel $z_2$ [l] at 350 10
\endlabellist
\centerline{\includegraphics[width=5.5in]{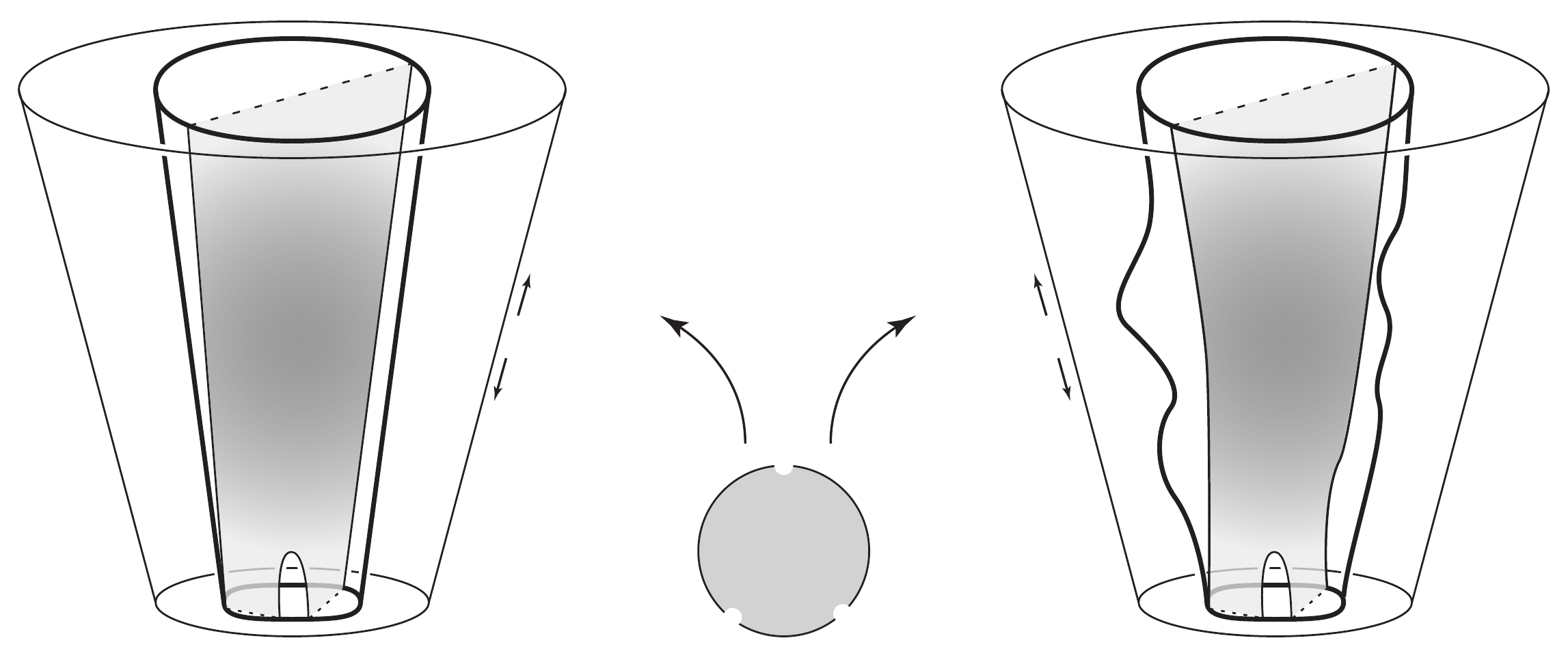}}
\caption{Schematic pictures of (a) a disk in the moduli space $\mathcal{M}^J_{\rr \times \leg}(a;\mathbf{b})$ and (b) a disk in  the moduli space $\mathcal{M}^J_{L}(a;\mathbf{b})$. }
\label{fig:j-disk}
\end{figure}
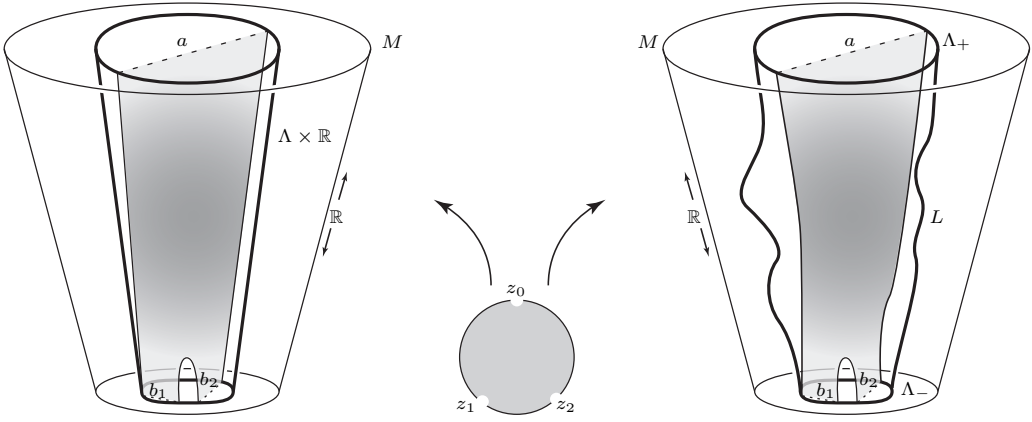
The differential of a Reeb chord $a \in \mathcal{R}_\leg$ counts $1$-dimensional moduli spaces, with that one dimension coming from translation invariance:
\begin{equation} \label{eq:lch-d}
\partial_\leg(a) = \sum_{ \dim \mathcal M^J_{\rr \times \leg} (a; \mathbf b) = 1} \# (\mathcal M^J_{\rr \times \leg} (a; \mathbf b)/\rr) \, \mathbf{b}.
\end{equation}
The differential is then extended to all of $\mathcal A_\leg$ via the Leibniz rule
and linearity.  

That the sum in Equation~(\ref{eq:lch-d}) is finite follows from the compactness of the moduli space and from an estimate on the \dfn{$\leg$-energy} of an element $u \in \mathcal{M}^J_{\rr \times \leg}$, defined to be
\[E_\leg(u) = \int_{D^2} u^*d\alpha.\]
The compatibility of $J$ with $d(e^t\alpha)$ and Stokes' Theorem imply:

\begin{lem} \label{lem:leg-energy}
	For a $J$-holomorphic disk $u \in M^J_{\rr \times \leg}(a; \mathbf{b})$, the $\leg$-energy satisfies
	\[ 0 \leq E_\leg(u) = h(a) - \sum h(b_i).\]
\end{lem}

The facts that the differential satisfies $\partial^2=0$ and that the homology of the DGA is invariant under Legendrian isotopy follows from a compactification and gluing argument \cite{ees:high-d-analysis,ees:pxr}.  

\subsection{Linearized LCH and its $A_\infty$ Structure}

In order to extract readily computable information from the Legendrian Contact Homology DGA, we use Chekanov's technique of linearization from \cite{chv}, enhanced to a full $A_\infty$ algebra as in \cite{products}.     Note that linearization is not always possible \cite{fuchs:augmentations, fuchs-ishk, rulings} and that different linearizations may lead to different linearized (co)homologies \cite{melvin-shrestha}.

We begin the linearization process by introducing some notation for the differential. The differential is determined by its action on $A_\leg$, and its restriction to $A_\leg$ may be written as a sum $\partial = \sum_{i=0}^\infty \partial_i$ with $\partial_i$ mapping $A_\leg$ to $A_\leg^{\otimes i}$.  In particular, the obstruction to the linear part of the differential, $\partial_1$, being a differential on $A_\leg$ is $\partial_0$. The goal of the linearization technique is to change coordinates so that $\partial_0 = 0$.   

The key idea in Chekanov's linearization technique is an \dfn{augmentation},  a unital DGA map $
\varepsilon: (\mathcal A_\leg, \partial_\leg) \to (\ff_2, 0)$ supported in degree $0$. Given an 
augmentation $\varepsilon$ of $(\mathcal{A}_\leg, \partial_\leg)$, one can construct a change of 
coordinates on $\mathcal{A}_\leg$ by $\eta^\varepsilon(a) = a + \varepsilon(a)$ and a new 
differential $\partial^\varepsilon = \eta^\varepsilon \partial_\leg (\eta^\varepsilon)^{-1}$.  It is easy to 
check that $\partial^\varepsilon_0 = 0$. We use the adjoint $d^\varepsilon$ of $\partial^
\varepsilon_1$ to define a differential on the the dual vector space $A^*_\leg$.   The cohomology 
groups of $(A^*_\leg, d^\varepsilon)$ are denoted $LCH^*(\leg, \varepsilon)$ and referred to as the 
\dfn{Linearized Legendrian  Contact Cohomology  (with respect to $\varepsilon$)}.  One may 
similarly define homology groups $LCH_*(\leg, \varepsilon)$ using the differential $\partial^
\varepsilon_1$, though we will not use these groups except in Example~\ref{ex:fund-class}, below.

At least some of the non-linear information of $\partial_\leg$ may be recovered in the linearized setting by introducing an $A_\infty$ algebra structure on $A^*_\leg$.  That is, we define a sequence of degree $1$, linear maps \[\mathbf{m}^\varepsilon = \{m_k^\varepsilon: (A_\leg^*)^{\otimes k} \to  A^*_\leg\}_{k \geq 1}\] 
to be the adjoints of the maps $\partial^\varepsilon_k$.

As shown in \cite{products}, the fact that $\partial_\leg^2 = 0$ implies that the sequence $\mathbf{m}^\varepsilon$ satisfies the $A_\infty$ relations
\[\sum_{i+j+k=n} m^\varepsilon_{i+1+k} \circ (1^{\otimes i} \otimes m^\varepsilon_j \otimes 1^{\otimes k}) = 0,\]
 for all $n \in \nn$.  In particular, $m_1^\varepsilon$ is a differential on $A^*_\leg$ (in fact, it is $d^\varepsilon$).  It is a standard fact that $\mathbf{m}^\varepsilon$ descends to an $A_\infty$ structure $\boldsymbol{\mu}^\varepsilon$ on the cohomology $LCH^*(\leg, \varepsilon)$ with $\mu^\varepsilon_1 = 0$ and with $\mu_2^\varepsilon$ giving an associative product; see \cite{kadeishvili}.

\begin{rem} \label{rem:lch-direct-defn}
 One can also define the codifferential operator $d^\varepsilon$ directly on the generators by:
\begin{equation} \label{eq:codiff}
d^\varepsilon(c) = \sum_{\dim \mathcal M^J_{\rr \times \leg} (a; \mathbf{b}c\mathbf{d}) = 1} \#(\mathcal M^J_{\rr \times \leg} (a; \mathbf b c \mathbf d)/\rr) \, \varepsilon(\mathbf b) \varepsilon(\mathbf d) a.
\end{equation}
Note, in particular, that the codifferential goes ``up'' the symplectization, while the differential $\partial_\leg$ goes ``down.''

Similarly, one may define the $A_\infty$ operations directly:
\begin{multline*}
m_k^\varepsilon(c_1, \ldots,  c_k) = \\ 
\sum_{\dim \mathcal M^J_{\rr \times \leg} (\cdots) = 1} \#(\mathcal M^J_{\rr \times \leg} (a; \mathbf{b}_0 c_1 \mathbf{b}_1 \cdots \mathbf{b}_{k-1} c_k \mathbf{b}_k)/\rr) \, \varepsilon(\mathbf{b}_0) \cdots \varepsilon(\mathbf{b}_k) a.
\end{multline*}
\end{rem}

\begin{rem}
	The grading convention in the definition of $A_\infty$ algebra above is not standard:  one usually wants $m_k$ to be a map of degree $2-k$.  As remarked in \cite{products} and as carried out in \cite{bc:bilinear}, one could fix this by working with the grading-shifted complex $A_\leg^*[1]$.  Further, we avoid the usual sign conventions for the $A_\infty$ relation as we are working over $\ff_2$. See \cite{fooo, seidel:fukaya} for a discussion of the standard grading and sign conventions for $A_\infty$ algebras.
\end{rem}

\begin{rem}
	Following \cite{bc:bilinear}, one can use \emph{two} augmentations $\varepsilon_1$ and $\varepsilon_2$ to  define the \dfn{Bilinearized Legendrian Contact (Co)homology}.  The underlying vector space is the same as above, but  the definition of the codifferential in Equation~(\ref{eq:codiff}) replaces 
 the terms $\varepsilon(\mathbf b) \varepsilon(\mathbf d)$ with $\varepsilon_1(\mathbf b) \varepsilon_2(\mathbf d)$.   Bourgeois and Chantraine show that the Bilinearized Legendrian Contact Homology may be enriched into an $A_\infty$ category called the \dfn{Augmentation Category}. 
 
All of the constructions developed later in this section also apply to the Bilinearized Legendrian 
Contact Cohomology and the Augmentation Category without significant change,   though we stay in 
the $A_\infty$ algebra context for ease of notation. The definition of the augmentation category may 
be adjusted to a \emph{unital} augmentation category \cite{nrssz:aug-sheaf}, but the original augmentation 
category is more suitable for the capacities to be defined below as it has a fundamental class instead 
of a unit.\end{rem}

\subsection{Capacities for Legendrian Submanifolds}
\label{ssec:leg-capacity}

In this subsection, we will associate to a fixed Legendrian submanifold $\leg$ and an augmentation $\varepsilon$ of its Legendrian Contact Homology DGA a set of  elements from
 $(0, \infty]$ called capacities.  These quantities are similar to capacities defined by Viterbo \cite{viterbo:generating} for a Lagrangian submanifold and by the authors \cite{josh-lisa:cap} for slices of a ``flat at infinity'' Lagrangian.  Keep in mind that these capacities are attached to a fixed Legendrian submanifold, and are \emph{not} invariant under Legendrian isotopy.  

The first step is to define a filtration on $A_\leg$. For any $w \in \rr$, define the sets 
\[\mathcal{R}_\leg^w = \{ \gamma \in \mathcal R_{\leg}\; : \; h(\gamma)  \geq w\},\]
where, as in (\ref{eqn:height}), $h(\gamma) = \int_\gamma \alpha.$
Let $F^wA^*_\leg$ be the graded vector space with basis $\mathcal{R}_\leg^w$.  By Lemma~\ref{lem:leg-energy}, $F^wA^*_\leg$ is a subcomplex of $A^*_\leg$ with respect to the differential $d^\varepsilon$ for any augmentation $\varepsilon$.  Hence the filtration descends to the Linearized Legendrian Contact Cohomology to define {\bf Filtered Linearized Legendrian Contact Cohomology},
  $LCH_w^*(\leg, \varepsilon)$,  the cohomology groups of the quotient $A^*_\leg / F^w A^*_\leg$.

\begin{rem}
	The filtration $F^wA^*_\leg$ is different from the one defined on the $A_\infty$ algebra associated to a Lagrangian in \cite[\S3.2]{fooo}, as we have no need for Novikov ring coefficients.
\end{rem}

Consider the projection $p^w: A^*_\leg \to A^*_\leg / F^w A^*_\leg$.  We may use $p^w$ to calculate the height of a Reeb chord $b$:
\[h(b) = \sup\{w : p^w(b) = 0\}.\]
The capacity of a cohomology class is simply an extension of this idea of height.
 Let $P^w: LCH^*(\leg, \varepsilon) \to LCH_w^*(\leg, \varepsilon)$ denote the map induced by $p^w$.  For $w$ close to zero, $P^w$ is the zero map, while for sufficiently large $w$, $P^w$ is an isomorphism.  Thus, for nonzero $\theta$, the set $\{w : P^w(\theta) = 0\}$ is non-empty and bounded above.

\begin{defn} \label{defn:capacity}
	Given a Legendrian submanifold $\leg \subset J^1M$, an augmentation $\varepsilon$ of its Legendrian Contact Homology DGA, and 
 $\theta \in LCH^*(\leg, \varepsilon)$, the \dfn{capacity} $c(\leg, \varepsilon, \theta)$ is defined to be:
$$ c(\leg, \varepsilon, \theta) =
	\begin{cases}
	  \sup \{w : P^w(\theta) = 0\}, & \theta \neq 0 \\
	  \infty, & \theta =0.
	\end{cases}
	$$
\end{defn}

A key fact for computing capacities is the following:

\begin{lem} \label{lem:cap-Reeb-length} For nonzero $\theta$,  $c(\leg, \varepsilon, \theta)$ is always the height of some Reeb chord.
\end{lem} 

The proof of this lemma is analogous to proof of Lemma 4.5 in \cite{josh-lisa:cap}.  

\begin{ex} \label{ex:unknot-computation}
Let us compute capacities for the Legendrian unknot $U(v)$ and its higher-dimensional generalizations  $U^n(v) \subset \rr^{2n+1}$.  The submanifold $U^n(v)$ has a single Reeb chord $\gamma$ of grading $n$ and height $v$.  For $n \geq 2$, it follows for grading reasons that $\partial \gamma = 0$, and the same result holds for $n=1$ if one computes as in \cite{chv}. Thus, we may use the trivial augmentation $\varepsilon$, and we can see that $LCH^*(U^n(v), \varepsilon)$ is generated by $[\gamma]$ in degree $n$.  Since $p^w(\gamma)$ vanishes for $w < h(\gamma) = v$ and is nonzero for $w > v$, we see that
	\[c(U^n(v), \varepsilon, [\gamma]) = v.\]
\end{ex}

\begin{ex} \label{ex:fund-class}
	For a connected Legendrian submanifold $\leg \subset \rr^{2n+1}$ with augmentation $\varepsilon$, there is a distinguished class $\lambda \in LCH^n(\leg, \varepsilon)$ called the \dfn{fundamental class}.\footnote{The fundamental class in this paper is hom-dual to the usual definition of the fundamental class of \cite{high-d-duality, duality}.  See below.}  To define $\lambda$, we use the duality exact sequence of \cite{high-d-duality}:
	\[ \xymatrix{
		\cdots \ar[r] & LCH_0(\leg, \varepsilon) \ar[r]^-{\rho_*} & H_0(\leg) \ar[r]^-{\sigma_*} & LCH^n(\leg, \varepsilon) \ar[r] & \cdots
	}\]
Theorem 5.5 of \cite{high-d-duality} shows that $\rho_*$ is trivial in degree $0$, and hence $\sigma_*$ is injective.  Define the fundamental class $\lambda \in LCH^n(\leg, \varepsilon)$ to be the image of the generator $[pt]$ of $H_0(\leg)$ (remember that we are using $\ff_2$ coefficients, so there is only one generator).  The usual definition of a fundamental class $\eta \in LCH_n(\leg, \varepsilon)$ is similar:  the map $\rho_*$ is onto in degree $n$, and $\eta$ has the property that $\rho_*(\eta) = [\leg]$, where the latter is the fundamental class in $H_n(\leg)$; note that this is equivalent to the definition in \cite[\S7.4]{ehk:leg-knot-lagr-cob} once one unwinds the definition of $\rho_*$. Using Theorem 1.1 of \cite{high-d-duality}, we see that
\begin{equation} \label{eq:fund-classes}
	\langle \lambda, \eta \rangle = \langle PD[pt], [\leg] \rangle = 1,
\end{equation}
where $\langle, \rangle$ is the usual pairing between homology and cohomology and $PD$ is Poincar\'e duality on $\leg$.

We define the \dfn{fundamental capacity} of $\leg$ with respect to the augmentation $\varepsilon$ to be $c(\leg, \varepsilon, \lambda)$.  The computation in Example~\ref{ex:unknot-computation}, above, is an example of a fundamental capacity.
\end{ex}

\section{Induced Maps from Lagrangian Cobordism} 
\label{sec:cobord-maps}

In this section, we will establish a lower bound for the length of a Lagrangian cobordism between two fixed Legendrians.  The bound arises from understanding 
induced maps between the filtered Linearized Legendrian Contact Cohomology of the Legendrians at the ends.  
   In particular, we   
establish the lower bounds for a Lagrangian cobordism
stated  in Theorem~\ref{thm:mono}, which is  generalized   to an $A_\infty$ version in Theorem~\ref{thm:a-infty-mono}.

\subsection{DGA cobordism maps}

That Legendrian Contact Homology has the structure of a field theory with respect to Lagrangian cobordism goes back to its initial conception \cite{egh}; the analytic details were eventually worked out in \cite[\S B]{ekholm:rsft} and \cite[\S 11]{compactness} and gathered together in \cite[\S 3]{ehk:leg-knot-lagr-cob}.
While the theory described in \cite[\S3]{ehk:leg-knot-lagr-cob} specifies $\rr^3$ as the ambient contact manifold, the analysis underlying the results holds in the more general case of $J^1M$.   In this section, we sketch the constructions of maps induced by a Lagrangian cobordism; summarizing sketches of the theory with goals similar to ours appear in \cite{c-dr-g-g, rizell:lifting, ehk:leg-knot-lagr-cob}. 

To fit Legendrian Contact Homology into a field theory, we must define a DGA map associated to a Lagrangian cobordism $L$ from
$\leg_-$ to $\leg_+$:
\[ \phi_L: (\mathcal{A}_{\leg_+}, \partial_{\leg_+}) \to (\mathcal{A}_{\leg_-}, \partial_{\leg_-}).\]
To define the map $\phi_L$, we introduce a new moduli space of $J$-holomorphic disks, where $J$ is an almost complex structure on the symplectization that is sufficiently generic, compatible, and cylindrical at the ends.  Given $a \in \mathcal{R}_{\leg_+}$ and a monomial $\mathbf{b} = b_1 \cdots b_k$ in $\mathcal{A}_{\leg_-}$, consider the set of $J$-holomorphic maps $u: D_k \to \rr \times J^1M$ that map the boundary of $D_k$ to $L$ and that satisfy asymptotic conditions similar to those in Section~\ref{sec:capacity}; see Figure~\ref{fig:j-disk}(b).  The moduli space $\mathcal{M}^J_L(a;\mathbf{b})$ is the set of such maps up to reparametrization, and it is, generically, a compact manifold \cite{ees:high-d-analysis,ees:pxr}.
The map $\phi_L$ may now be defined on the generators of $\mathcal{A}_{\leg_+}$ by:
\begin{equation} \label{eq:cob-map}
\phi_L(a) = \sum_{\dim \mathcal{M}^J_L(a;\mathbf{b}) = 0} \# \mathcal{M}^J_L(a;\mathbf{b}) \, \mathbf{b}.
\end{equation}

In Section~\ref{sec:capacity}, we explained that the finiteness of the sum in $\partial_\leg(a)$ follows from compactness and an examination of the $\Lambda$-energy of a $J$-holomorphic disk.  To prove that the sum in Equation~(\ref{eq:cob-map}) is finite, we use a notion of ``Lagrangian energy," which is a  
close relative to 
an energy introduced in \cite{c-dr-g-g}.   Given a Lagrangian cobordism $L$ that is cylindrical
outside $[s_-, s_+]$, consider the piecewise smooth function $\varphi: \rr \to [e^{s_-}, e^{s_+}]$
defined by
\[\varphi(t) = 
\begin{cases}
e^{s_-}, & t \leq s_-\\
e^t, &t \in [s_-, s_+] \\
e^{s_+}, & t \geq s_+.
\end{cases}
\]
The \dfn{$L$-energy of  a $J$-holomorphic disk} $u: (D_k, \partial D_k) \to (\rr \times J^1M, L)$
is defined to be:
\[E(u) = \int_D u^* d(\varphi \alpha).\]
 The $L$-energy depends on the choice of $s_\pm$, though this will not matter in applications.

To better understand the relationship between the $L$-energy and the cobordism map $\phi_L$, we introduce the ``{cobordism action}" of a Reeb chord $\gamma$.  Suppose that a Lagrangian cobordism $L$ is cylindrical
outside of $[s_-, s_+]$, where $s_\pm$ are the same constants chosen for the definition of the $L$-energy. We then define the {\bf cobordism action} of a Reeb chord of $\leg_\pm$ to be scaled heights: 
\begin{equation}\label{eqn:cob-action}
\a(\gamma) = 
\begin{cases}
e^{s_+} h(\gamma), & \gamma \in \mathcal R_{\leg_+}\\
e^{s_-} h(\gamma), & \gamma \in \mathcal R_{\leg_-}.
\end{cases}
\end{equation}

As in Lemma~\ref{lem:leg-energy}, we may relate the $L$-energy of a $J$-holomorphic disk in $\mathcal{M}^J_L$ to the cobordism actions of the Reeb chords at its ends.

\begin{lem}[\cite{c-dr-g-g}] \label{lem:disk-energy} Given a Lagrangian cobordism $L$ from $\leg_-$ to $\leg_+$ that is cylindrical outside of $[s_-,s_+]$, 
the $L$-energy of $u \in \mathcal M_L^J(a; \mathbf b)$ satisfies
$$ 0 \leq E(u) = \a(a) - \sum_i \a(b_i).$$
\end{lem}

The finiteness of the sum in Equation~(\ref{eq:cob-map}) follows. That the map $\phi_L$ is, indeed, a chain map follows from compactness and gluing arguments in \cite[\S B]{ekholm:rsft} or in \cite[\S 11]{compactness}.  In all of our examples, the Lagrangian cobordisms will be diffeomorphic to cylinders over Legendrians with vanishing rotation number, and hence will have Maslov number zero.  Thus, we find ourselves in the setting of \cite{ekholm:rsft}, and the cobordism map has degree $0$.  Finally, if the cobordisms $L$ and $L'$ differ by an exact Lagrangian isotopy, then $\phi_L$ and $\phi_{L'}$ differ by a chain homotopy \cite{ekholm:rsft, ehk:leg-knot-lagr-cob}.

\subsection{$A_\infty$ cobordism maps}
\label{ssec:lin-cob}

 The material in this section has not been written down explicitly before, but its existence should be clear to those familiar with the theory.
 
In the presence of a Lagrangian cobordism $L$ from $\leg_-$ to  $\leg_+$, we simplify notation by 
writing $\mathcal{A}_\pm$ for $\mathcal{A}_{\leg_\pm}$, etc. Given an augmentation $\varepsilon_-$ 
of $(\mathcal{A}_-, \partial_-)$, the Lagrangian cobordism induces an augmentation $\varepsilon_+ = 
\varepsilon_- \circ \phi_L$ on $(\mathcal{A}_+, \partial_+)$.  We may use these augmentations to 
change coordinates and define a new cobordism map $\phi_{L, \varepsilon_-}: (\mathcal{A}_+, 
\partial^{\varepsilon_+}) \to (\mathcal{A}_-, \partial^{\varepsilon_-})$. If we decompose the restriction 
of $\phi_{L, \varepsilon_-}$ to $A_+$ by components of $TA_- = \bigoplus_{i = 0}^\infty A_-^{\otimes i}$, it is straightforward to check that the constant term vanishes.\footnote{This is dual to the results of 
\cite[\S5.2]{fooo} on removing obstructions to $A_\infty$ maps being ``strict'' using bounding 
cochains.} By taking the adjoints of the components of $\phi_{L, \varepsilon_-}$, we obtain a (degree 
$0$) $A_\infty$ map $\boldsymbol{\psi}^{L,\varepsilon_-}: (A_-^*, \mathbf{m}^-) \to (A_+^*, \mathbf{m}^+)$.  Recall that an $A_\infty$ map is really a sequence of maps
 $\psi^{L, \varepsilon_-}_k: (A_-^*)^{\otimes k} \to A_+^*$ that satisfy: 
\begin{equation} \label{eq:a-infty-map}
	\sum_{i+j+k=n} \psi^{L,\varepsilon_-}_{i+1+k} \circ (1^{\otimes i} \otimes m^-_j \otimes 1^{\otimes k}) = \sum_{\substack{1 \leq r \leq n \\ i_1 +\cdots +i_r = n}} m^+_r \circ (\psi^{L, \varepsilon_-}_{i_1} \otimes \cdots \otimes \psi^{L, \varepsilon_-}_{i_r}).
\end{equation}
\begin{rem} \label{rem:cob-direct-defn}
	As in Remark~\ref{rem:lch-direct-defn}, we may directly define the sequence of
	maps making up the $A_\infty$ map $\boldsymbol{\psi}^{L,\varepsilon_-}$ as follows:
	\begin{multline*}
\psi^{\varepsilon_-}_k(c_1, \ldots,  c_k) = \\ 
\sum_{\dim \mathcal M^J_{L} (\cdots) = 0} \#\mathcal M^J_{L} (a; \mathbf{b}_0 c_1 \mathbf{b}_1 \cdots \mathbf{b}_{k-1} c_k \mathbf{b}_k) \, \varepsilon_-(\mathbf{b}_0) \cdots \varepsilon_-(\mathbf{b}_k) a.
\end{multline*}
\end{rem}

The $A_\infty$ map $\boldsymbol{\psi}^{L, \varepsilon_-}$ induces a map on homology denoted $\boldsymbol{\Psi}^{L, \varepsilon_-}$.  In particular, we have a sequence of maps:
\begin{equation}
\Psi_k^{L, \varepsilon_-}: LCH^*(\leg_-, \varepsilon_-)^{\otimes k} \to LCH^*(\leg_+, \varepsilon_+).
\end{equation}

\subsection{Non-triviality of the Cobordism Map }
\label{ssec:non-trivial}

In our applications in Section~\ref{sec:applications}, we will need use classes in $LCH^*(\leg_-, \varepsilon_-)$ that are not in the kernel of
$\Psi_k^{L,\varepsilon_-}$.   The following proposition, which essentially re-contextualizes results from \cite{c-dr-g-g-cobordism} and \cite{ehk:leg-knot-lagr-cob}, will be used in Section~\ref{sec:applications}.

\begin{prop}\label{prop:non-ker}  Suppose $L$ is a Lagrangian cobordism from $\leg_-$ to $\leg_+$ that is cylindrical outside $[s_-, s_+]$ and  $\leg_-$ has
augmentation $\varepsilon_-$.  Let 
$\Psi_1^{L, \varepsilon_-}: LCH^*(\leg_-, \varepsilon_-) \to LCH^*(\leg_+, \varepsilon_+)$
be the associated cobordism map.  
\begin{enumerate}
\item  If $\leg_\pm$ are connected and  $\lambda_- \in LCH^*(\leg_-, \varepsilon_-) $ is the fundamental class, then
$\Psi_1^{L, \varepsilon_-}(\lambda _-) \neq 0$.
\item If  $\bar{L}$ denotes $L \cap \{ s \in [s_-, s_+] \}$ and the pair $(\bar{L}, \leg_-)$ is acyclic, then
$\Psi_1^{L, \varepsilon_-}$ is an isomorphism.
\end{enumerate}
\end{prop}

\begin{proof}
  Let $\Phi^1_{L, \varepsilon_-}: LCH_*(\leg_+, \varepsilon_+) \to LCH_*(\leg_-, \varepsilon_-)$ denote the adjoint of $\Psi_1^{L,\varepsilon_-}$.  
  Under the hypothesis that $\leg_+$ is connected, we can consider its (homology) fundamental class, $\eta_+ \in LCH_n(\leg_-, \varepsilon_-)$.  As shown in \cite[Theorem 7.7]{ehk:leg-knot-lagr-cob},  $\Phi^1_{L, \varepsilon_-}(\eta_+)$ is the homology fundamental class in $LCH_n(\leg_-, \varepsilon_-)$.  Using Equation~(\ref{eq:fund-classes}), we have:
\begin{align*}
	1 &= \langle \lambda_-, \Phi^1_{L, \varepsilon_-}(\eta_+) \rangle \\
	&= \langle \Psi_1^{L, \varepsilon_-}(\lambda_-), \eta_+ \rangle.
\end{align*}
In particular, the fundamental class $\lambda_-$ is not in the kernel of $\Psi^{L, \varepsilon_-}_1$.  

Statement (2) follows immediately from a long exact sequence of \cite[Theorem 1.6]{c-dr-g-g-cobordism}, which relates  relative homology groups of a Lagrangian
cobordism to the linearized
Legendrian Contact Cohomology groups of its ends:
 \[ \xymatrix@=15pt{
		\cdots \ar[r] & H_{n+1-k}\left(\bar{L}, \leg_- \right) 
		 \ar[r] & LCH^k (\leg_-, \varepsilon_-) \ar[r]^-{\Psi_1^{L, \varepsilon_-}} & LCH^k(\leg_+, \varepsilon_+) \ar[r] & \cdots .
	}\]
\end{proof}

\subsection{Cobordisms maps on filtered complexes}
\label{ssec:lagr-capacity}

We now have the tools  to relate the capacities at the top and bottom of a Lagrangian cobordism.  We will prove Theorem~\ref{thm:a-infty-mono}, from which Theorem~\ref{thm:mono} will follow as an immediate corollary.

The first step in the proof is to show that the $A_\infty$ cobordism maps $\boldsymbol{\psi}^{L,\varepsilon_-}$ respect the filtrations on $A^*_\pm$.
 
\begin{lem}  \label{lem:filtered-map}  Suppose $L$ is a Lagrangian cobordism
from $\leg_-$ to $\leg_+$ that is cylindrical outside of $[s_-, s_+]$; let
$\psi^{L, \varepsilon_-}_k: (A_-^*)^{\otimes k} \to A_+^*$ be the associated $A_\infty$ maps.
	The image of $F^{w_1/e^{s_-}}A^*_- \otimes \cdots \otimes F^{w_k/e^{s_-}}A^*_-$ under $\psi^{L,\varepsilon_-}_k$ lies in $F^{(w_1 + \cdots + w_k)/e^{s_+}}A^*_+$.
\end{lem}

\begin{proof}  Recall that $F^{w_i/e^{s_-}}A^*_-$ is generated by the Reeb chords $b$
satisfying $h(b) \geq w_i/e^{s_-}$.  Thus it suffices to show that 
given a set of Reeb chords $\{b_1, \ldots, b_k\}$ with 
$b_i \in F^{w_i/e^{s_-}}A^*_-$, if $a$ appears as a term in  
$\psi^{L,\varepsilon_-}_k(b_1, \ldots, b_k)$, then $h(a)  \geq 
(w_1 + \cdots + w_k)/e^{s_+}$.  Given such $\{b_1, \ldots, b_k\}$ and $a$, 
 Lemma~\ref{lem:disk-energy} and the formula for the map $\psi^{L,\varepsilon_-}_k$ in Remark~\ref{rem:cob-direct-defn} show that
\begin{equation} \label{eq:filtered-action}
	\mathfrak{a}(a) \geq \sum_i \mathfrak{a}(b_i).
\end{equation}
We now use the definition of the cobordism action $\mathfrak{a}$, Equation~(\ref{eq:filtered-action}), and the fact that $h(b_i) \geq w_i/e^{s_-}$ to see that :
\begin{align*}
	h(a) &= e^{-s_+} \mathfrak{a} (a) \\
		&\geq e^{-s_+} \sum_i \mathfrak{a}(b_i) \\
		&= e^{s_- - s_+} \sum_i h(b_i) \\
		&\geq e^{-s_+} \sum_i w_i.
\end{align*}
It follows that  each summand of  $\psi^{L,\varepsilon_-}_k(b_1, \ldots, b_k)$ lies in $F^{(w_1 + \cdots + w_k)/e^{s_+}}A^*_+$, as desired.
\end{proof}

We are now ready to state and prove a more general form of Theorem~\ref{thm:mono}.

\begin{thm} \label{thm:a-infty-mono} Suppose $L$ is a  Lagrangian cobordism from $\leg_-$ to $\leg_+$ that is cylindrical outside of $[s_-,s_+]$.  For any  augmentation $\varepsilon_-$ of $(\mathcal{A}_-, \partial_-)$, let $\varepsilon_+ = \varepsilon_- \circ \phi_L$ denote the corresponding augmentation of $(\mathcal{A}_+, \partial_+)$.
Then, for any $\{\theta_1, \ldots, \theta_k\} \subset LCH^*(\leg_-, \varepsilon_-)$,  
	\[e^{s_-} \sum_i c\left(\leg_-,\epsilon_-,\theta_i\right) \leq e^{s_+}c\left(\leg_+, \epsilon_+, \Psi^{L,\varepsilon_-}_k(\theta_1, \ldots, \theta_k)\right).\]
In particular,  for any augmentation $\varepsilon_-$ of $\leg_-$ and any $\theta_1, \dots, \theta_k \in LCH^*(\leg_-, \varepsilon_-)$ with $\Psi_k^{L,\varepsilon_-}(\theta_1, \dots, \theta_k) \neq 0$, a lower bound to the length of $L$ is given by: 
\begin{equation} \label{eq:length-mono-k}
		\ell(L) \geq  \sum_{i=1}^k \ln c\left(\leg_-, \varepsilon_-,\theta_i\right) - \ln c \left(\leg_+, \varepsilon_+, \Psi_k^{L,\varepsilon_-}(\theta_1, \dots, \theta_k)\right).
\end{equation}
\end{thm}

\begin{proof} 
We first consider the following commutative diagram, where the right-hand vertical arrow exists by Lemma~\ref{lem:filtered-map}:
\[
	\xymatrix@C=70pt{
		LCH^*(\leg_+,\varepsilon_+) \ar[r]^-{P^{(w_1 + \cdots + w_k)/e^{s_+}}} & LCH^*_{{(w_1 + \cdots + w_k)/e^{s_+}}}(\leg_+, \varepsilon_+) \\
		\bigotimes_{i=1}^k LCH^*(\leg_-, \varepsilon_-) \ar[u]^{\Psi^{L,\varepsilon_-}_k} \ar[r]^-{\bigotimes P^{w_i/e^{s_-}}} & \bigotimes_{i=1}^k LCH^*_{w_i / e^{s_-}  }(\leg_-, \varepsilon_-) \ar[u]_{\Psi^{L,\varepsilon_-}_k}
	}
\]

Next, we rewrite the definition of capacity to help take the cobordism action into account:
\begin{equation} \label{eq:cobord-capacity}
\begin{split}
	e^{s} c(\leg, \varepsilon, \theta) &= e^{s} \sup \{ w / e^{s} : P^{w / e^{s}}(\theta) = 0\} \\
	&= \sup \{w : P^{w / e^{s}}(\theta) = 0\}.
	\end{split}
\end{equation}

Fix $\{\theta_1, \ldots, \theta_k\} \subset LCH^*(\leg_-, \varepsilon_-)$.  The commutativity of the diagram above shows that if $P^{w_i / e^{s_-}}(\theta_i) = 0$ for each $i \in \{1, \ldots, k\}$, then we have
\[P^{(w_1 + \cdots + w_k) / e^{s_+}}\left(\Psi^{L,\varepsilon_-}_k(\theta_1, \ldots, \theta_k)\right) = 0.\]
Equation~(\ref{eq:cobord-capacity}) then implies that  
\[e^{s_+} c\left(\leg_+, \varepsilon_+,  \Psi^{L,\varepsilon_-}_k(\theta_1, \ldots, \theta_k)\right) \geq \sum w_i,\]
for any $w_i$ with the property that $P^{w_i / e^{s_-}}(\theta_i) = 0$, and the theorem follows.
\end{proof}

\section{Applications}
\label{sec:applications}

Having established constructions of Lagrangian cobordisms in 
  Section~\ref{sec:constructions} and lower bounds to the length of a Lagrangian
  cobordism in Theorem~\ref{thm:a-infty-mono},
 we proceed to examine the examples outlined in the Introduction:  the vertical  contractions of Theorem~\ref{thm:unknot-dilation} in Section~\ref{ssec:dilation}, the Hopf links of Theorem~\ref{thm:hopf} in Section~\ref{ssec:hopf}, the packing problem of Question~\ref{ques:packing} in Section~\ref{ssec:packing}, and the loops of Legendrians of Theorem~\ref{thm:loops} in Section~\ref{ssec:loops}.  

\subsection{Reeb Chord Contractions}
\label{ssec:dilation}

We begin by re-examining the vertical contraction of a Legendrian $\leg^\sigma$ from Proposition~\ref{prop:scaling} in light of Theorem~\ref{thm:mono}.
In Proposition~\ref{prop:scaling}, we saw that when $\sigma<1$, there exist Lagrangian cobordisms from
$\leg_- = \leg$ to $\leg_+ = \leg^\sigma$ of length arbitrarily close to $\ln(1/\sigma)$. We will apply Theorem~\ref{thm:mono} to  
give lower bounds on the length of such a cobordism:

\begin{prop} \label{prop:chord-diff} Suppose $\leg_-$ is a connected Legendrian submanifold with an augmentation with the property that all its Reeb chords  have length at least $u$, 
and $\leg_+$  is a connected Legendrian submanifold such that all its Reeb chords  have length at most $v$.  If $L$ is a Lagrangian cobordism
from $\leg_-$ to $\leg_+$, then
\[\ell(L) \geq \ln \frac{u}{v}.\]
\end{prop}
 
\begin{proof} By Theorem~\ref{thm:mono} (i.e. Theorem~\ref{thm:a-infty-mono} with $k=1$), we know that for any $\theta \in LCH^*(\leg_-, \varepsilon_-)$ that is not in $\ker \Psi_1^{L,\varepsilon_-}$, 
\[
	\ell(L) \geq \ln c\left(\leg_-, \varepsilon_-, \theta\right) - \ln c\left(\leg_+,\varepsilon_+, \Psi_1^{L,\varepsilon_-}(\theta)\right).
\] 
 We also know, by Lemma~\ref{lem:cap-Reeb-length}, that as long as  $\Psi_1^{L,\varepsilon_-}(\theta)$ is nonzero, $$u\leq c(\leg_-,\varepsilon_-,\theta) \text{ and }
 c(\leg_+, \varepsilon_+, \Psi_1^{L,\varepsilon_-}(\theta)) \leq v.$$ 
 By Proposition~\ref{prop:non-ker}, our result follows by taking $\theta$ to be the fundamental class.
\end{proof}

Theorem~\ref{thm:unknot-dilation} follows immediately from Propositions~\ref{prop:scaling} and \ref{prop:chord-diff}.



\begin{rem}  For a connected Legendrian $\leg$ with an augmentation and with Reeb chords of lengths $h_1 \leq h_2 \leq \dots \leq h_n$,
for sufficiently small $\sigma$, namely $\sigma h_n < h_1$,  Proposition~\ref{prop:chord-diff}
will give an obstruction to arbitrarily short Lagrangian cobordisms from \leg\ to $\leg^\sigma$.  However, there is
a gap between  the upper and lower bounds to the length of a Lagrangian cobordism
given by, respectively,  Proposition~\ref{prop:scaling} and Proposition~\ref{prop:chord-diff}.   The challenge to getting better lower bounds  is to understand properties of 
 the cobordism map $\Psi_1^{L,\varepsilon_-}$ and the induced augmentation $\varepsilon_+ = 
\varepsilon_- \circ \phi_L$.  
      \end{rem}

\subsection{Lagrangians with Linked Boundaries}
\label{ssec:hopf}

For our second application, we show that  the length of a Lagrangian cobordism between two $2$-component links can differ significantly from the lengths of cobordisms between individual pairs of components.  

More specifically, let $U^n(1)$ denote a Legendrian submanifold with precisely
one Reeb chord of height $1$.  Form the Hopf link $H(\nu)$ of two components from this Legendrian
sphere and a vertical translate: 
\[H(\nu) = U^n(1) \cup (U^n(1))_\epsilon^\nu, \qquad 0 < \nu < 1,\]
where $(U^n(1))_\epsilon^\nu$ is a small horizontal dilation of the vertical translate so as to make the link chord generic.
We now consider Lagrangian cobordisms between $H(u)$ and $H(v)$, $0 < u,v < 1$. 
Recall that in Corollary~\ref{cor:vert-shift},  we showed that there always exists an
arbitrarily short Lagrangian cobordism between  $\leg_- = (U^n(1))_\epsilon^u$ and
its vertical translate $\leg_+ = (U^n(1))_\epsilon^v$, for any $u, v$.  
In particular, the vertical shifting between individual components of the Hopf links can be achieved with arbitrarily short cobordisms.  When considering cobordisms between links, however, Remark~\ref{rmk:vert-distort} hints that realizing vertical shifts in the presence of another cobordism could require a longer cobordism.  We will  apply the
the obstruction given by Theorem~\ref{thm:mono}  and the
constructions from Section~\ref{sec:constructions}  to show a higher dimensional version of Theorem~\ref{thm:hopf}:

\begin{thm}\label{thm:n-dim-hopf}
Any Lagrangian cobordism from $\leg_- = H(u)$ to $\leg_+ = H(v)$ composed of two Lagrangian cylinders that join the upper (resp. lower) component of $H(u)$ to the upper (resp. lower) component of $H(v)$ has length at least 
	\begin{itemize}
	\item $\ln \frac{1-u}{1-v}$ if $u\leq v$ and
	\item $\ln \frac{u}{v}$ if $u \geq v$,
	\end{itemize} 
	and the bounds are the best possible.
	\end{thm}

\begin{proof} We begin by proving the claimed lower bound on the length of a cobordism $L$ from $\leg_- = H(u)$ to $\leg_+ = H(v)$.  To do so, we find nontrivial classes in the Legendrian Contact Cohomology of $H(u)$ and compute their associated capacities.  

 We may assume that 
 $H(u)$ has exactly six Reeb chords, one between each pair of strands at the midpoint of the (spun) front diagram in Figure~\ref{fig:hopf-chords}. As discussed in Example~\ref{ex:unknot-computation}, each component of $H(u)$ has a (trivial) augmentation, and hence $H(u)$ itself does; call it $\varepsilon_u$.  With respect to this trivial augmentation, we use Mishachev's homotopy splitting \cite{kirill} to decompose the vector space $A_{H(u)}$ into four subcomplexes, one for each ordered pair of components; a Reeb chord is a generator of the subcomplex corresponding to its lower and upper ends. The chord $\delta_{1-u}$ depicted in Figure~\ref{fig:hopf-chords} is the sole generator of its subcomplex, and hence yields a cohomology class $[\delta_{1-u}]$.  The chord $\gamma_u$, on the other hand, is one of three chords generating its subcomplex, but the argument in \cite[\S6]{bst:construct}\footnote{While the calculation in \cite[\S6]{bst:construct} is ostensibly for generating family homology, the same argument works for Legendrian contact homology as only length and grading of Reeb chords are used.} shows that it must yield a cohomology class $[\gamma_u]$.
It follows immediately from Lemma~\ref{lem:cap-Reeb-length} that the capacities of these classes are:
\begin{equation} \label{eq:hopf-cap}
	\begin{split}
	c(H(u), \varepsilon_u, [\gamma_u]) &= u,  \\
	c(H(u), \varepsilon_u, [\delta_{1-u}]) &= 1-u.
\end{split}
\end{equation}

\begin{figure}
\labellist
\small
\pinlabel $\gamma$ [l] at 44 20
\pinlabel $\delta$ [l] at 44 45
\endlabellist

\centerline{\includegraphics{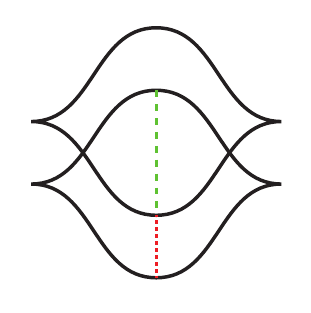}}
\caption{The chords $\gamma_u$ and $\delta_{1-u}$ each yield a non-zero class in $LCH^*(H(u), \varepsilon)$.}
\label{fig:hopf-chords}
\end{figure}

Suppose $L$ is a Lagrangian cobordism from $\leg_- = H(u)$ to $\leg_+ = H(v)$.  To apply Theorem~\ref{thm:mono},  we begin by asserting that it is clear that the cobordism 
map $\Psi_1^{L, \varepsilon_u}$ preserves Mishachev's homotopy splitting.  Combining this fact with Proposition~\ref{prop:non-ker}, we see that $\Psi_1^{L, \varepsilon_u}([\gamma_u]) = [\gamma_v]$ and that $\Psi_1^{L, \varepsilon_u}([\delta_{1-u}]) = [\delta_{1-v}]$.   Theorem~\ref{thm:mono} then gives the stated lower bounds:
\[\ell(L) \geq \max\left\{ \ln \frac{1-u}{1-v}, \ln \frac{u}{v}\right\} =
\begin{cases}
\ln \frac{1-u}{1-v}, & u < v\\
\ln \frac{u}{v}, & u > v.
\end{cases}
\]

To see that these bounds are the best possible, we apply Lemmas~\ref{lem:isotopy-cobord} and \ref{lem:reeb-cobord}.  Namely,
consider a smooth function $\rho(s)$ so that $\rho(s) = u$ for $s \leq 0$ and $\rho(s) = v$ for $s \geq A$. Parameterize $\leg$ by
$(x(t), y(t), z(t))$ for $t \in \Sigma$.  Let $\Sigma_1, \Sigma_2$ be disjoint  copies of $\Sigma$ and consider the Legendrian embedding 
 $\lambda_s(\Sigma_1 \cup \Sigma_2)$ given by
$$\lambda_s(t) = \begin{cases}
((1+\epsilon)x(t), y(t)/(1 + \epsilon), z(t) + \rho(s)), & t \in \Sigma_2 \\
(x(t), y(t), z(t)), & t \in \Sigma_1
\end{cases}. $$ 
It follows that 
$$\eta(s, t) - \eta(s, t') = 
\begin{cases}
-\rho'(s), &t \in \Sigma_1, t' \in \Sigma_2 \\
\rho'(s), & t \in \Sigma_2, t' \in \Sigma_1 \\
0, & t, t' \in \Sigma_i.
\end{cases}
$$

The Hopf link $\leg_s = \image \lambda_s$ has six Reeb chords.  By Lemma~\ref{lem:reeb-cobord}, it suffices to choose $\rho(s)$ so that
for each of these six Reeb chords, the height of the Reeb chord from $\lambda_s(t)$ to $\lambda_s(t')$
 does not equal  $\eta(s, t) - \eta(s, t')$.

Two Reeb chords  begin and end on the same component and have height $1$.  These Reeb chords give no obstruction to embedding since $1 \neq 0 = \eta(s, t) - \eta(s, t')$. 
 
There are two Reeb Chords of height $\rho(s)$ that come from $t \in \Sigma_1, t' \in \Sigma_2$.  
For these Reeb chords, the embedding condition is  $\rho(s) \neq - \rho'(s)$, or equivalently $\frac{d}{ds} (e^s \rho(s)) \neq 0$.  Thus, we are required to choose $\rho(s)$ so that $e^s \rho(s)$ is a strictly increasing function and so that
\begin{equation} \label{eq:hopf-condition}
	u = e^0 \rho(0) < e^A \rho(A) = e^A v.
\end{equation}
When $u < v$, we may find a $\rho(s)$ that satisfies Equation~\ref{eq:hopf-condition} for any $A>0$.  When $u > v$,  we may find a $\rho(s)$ that satisfies Equation~\ref{eq:hopf-condition} so long as $A > \ln{u}/{v}$.  

There will be one Reeb chord of height $1-\rho(s)$ coming from $t \in \Sigma_2, t' \in \Sigma_1$.  For this Reeb chord,
the embedding condition is $1 - \rho(s) \neq \rho'(s)$, or equivalently $\frac{d}{ds} (e^s (\rho(s) - 1  )) \neq 0$.  Since $\rho(s) - 1 < 0$, we now want to choose $\rho(s)$ so that
$e^s(\rho(s) - 1)$ is a strictly decreasing function.  If $u < v$,  finding an appropriate $\rho$ is possible as long as $A > \ln \frac{1-u}{1-v}$.   If $u > v$, we may find an appropriate $\rho(s)$ for any $A>0$.

The one remaining  Reeb chord has height $1 + \rho(s)$ and comes from $t \in \Sigma_1, t' \in \Sigma_2$. For this Reeb chord,
the embedding condition is: $1 + \rho(s) \neq -\rho'(s)$, or equivalently $\frac{d}{ds} (e^s (\rho(s) + 1  )) \neq 0$.   If $u < v$, $e^s(\rho(s) + 1)$ is
guaranteed to be a strictly increasing function   by
choosing $\rho(s)$ to be an increasing function.  When $u > v$, embeddedness can be guaranteed as long as
$(e^s \rho(s) + 1)|_{s = 0} < (e^s \rho(s)+1)|_{s = A}$, equivalently $A > \ln\frac{u+1}{v+1}$.  However, this is not a new restriction: we already know
that when $u > v$, we need $A > \ln \frac{u}{v}$, and since $\ln\frac{u}{v} > \ln \frac{u+1}{v+1}$, this Reeb chord introduces no new restrictions. 
\end{proof}

\subsection{Packing Lagrangian Cylinders}
\label{ssec:packing}

In the previous section, we examined the minimum length of a cobordism between two different Hopf links.  Now we turn this question
around and ask: if a Lagrangian cobordism has length $1$, which  links of unknots can appear as the boundaries? More precisely, we form a $(k+1)$-copy
of the Legendrian submanifold $U^n(1)$, where each copy is vertically translated by $v_i$ and then slightly dilated horizontally so that the resulting link,
$H(v_1, \ldots, v_k)$ is chord generic:
$$H(v_1, \ldots, v_k) = U^n(1) \cup (U^n(1))_{\epsilon_1}^{v_1} \cup \dots \cup  (U^n(1))_{\epsilon_k}^{v_k} \qquad 0 < v_1 < \dots < v_k  < 1,$$ 
When $v_i = \frac{i}{k+1}$, for each $i \in \{1, \ldots, k\}$, we call $H\left(\frac{1}{k+1}, \ldots, \frac{k}{k+1}\right)$ the \dfn{evenly shifted $(k+1)$-copy} of $U^n(1)$.

We return to the question asked in Question~\ref{ques:packing}:  fix the evenly shifted $(k+1)$-copy  as 
the negative boundary condition for the Lagrangian cobordism.  Which  
shifted $(k+1)$-copies $H(v_1, \dots, v_n)$  can appear as the  positive boundary 
of a Lagrangian cobordism of length one? See Figure~\ref{fig:packing}.

For the $2$-copy ($k=1$), the question reduces to:  for which $v$ does there exist a Lagrangian cobordism $L$ from $\leg_- = H(1/2)$ to $\leg_+ = H(v)$ with $\ell(L) \leq 1$?  This question was essentially answered in the previous section.   Theorem~\ref{thm:n-dim-hopf}  shows that such a cobordism exists if and only if
\[\frac{1}{2e} \leq v \leq 1 - \frac{1}{2e}.\]

Using similar techniques, we may answer the  question for arbitrary $k$:

\begin{thm} \label{thm:packing}  If there exists a Lagrangian cobordism  of length $1$ from the evenly spaced $(k+1)$-copy $H$ to
the shifted $(k+1)$-copy $H(v_1, \ldots, v_n)$ composed of $(k+1)$ Lagrangian cylinders joining corresponding components, then
	\[\frac{i-j}{(k+1)e} \leq v_i-v_j \leq 1- \frac{(k+1) - (i-j)}{(k+1)e},\]
	where $i\in \{1, \ldots, k\}$, $j \in \{0, \ldots, k-1\}$, $i>j$, and $v_0=0$.  These bounds are the best possible.
\end{thm}
In the case of the $3$-copy, these inequalities are visualized in Figure~\ref{fig:packing-ineq}.

\begin{proof} 
	The proof relies on  rewriting the last inequality in Theorem~\ref{thm:mono} in the case of a cobordism of length at most  $1$:
	\begin{equation} \label{eq:length-1-mono}
		c\left(\leg_+,\varepsilon_+,\Psi_1^{L,\varepsilon_-}(\theta)\right) \geq \frac{1}{e} c\left(\leg_-,\varepsilon_-, \theta\right).
	\end{equation}
	
	Fix the trivial augmentation $\varepsilon_-$ on $H$.  For each pair of components of the $(k+1)$-copy, we will apply Equation~(\ref{eq:length-1-mono}) to cohomology classes corresponding to $\delta$ and $\gamma$ in the proof of Theorem~\ref{thm:n-dim-hopf}.  As in the proof of Theorem~\ref{thm:n-dim-hopf}, 
	 Mishachev's homotopy splitting implies that  the computation of $LCH^*(H, \varepsilon_-)$ and $LCH^*(H(v_1, \ldots, v_n), \varepsilon_+)$ splits into computations entirely similar to that for $LCH^*(H(u),\varepsilon)$. \  In particular, we obtain two non-trivial classes $[\gamma^\pm_{ij}]$ and $[\delta^\pm_{ij}]$ for every pair of components $i>j$.
	
As before, the capacity of each generator is easy to compute from the length of the associated Reeb chord:
\begin{align*}
		c(H, \varepsilon_-, [\gamma^-_{ij}]) &= \frac{i-j}{k+1}, & c(H(v_1, \ldots, v_n), \varepsilon_+, [\gamma^+_{ij}]) &= v_i-v_j \\
		c(H, \varepsilon_-, [\delta^-_{ij}]) &= 1-\frac{i-j}{k+1}, & c(H(v_1, \ldots, v_n), \varepsilon_+, [\delta^+_{kl}]) &= 1-(v_i-v_j).
\end{align*}
Proposition~\ref{prop:non-ker} and the fact that $\Psi^{L, \varepsilon_-}_1$ preserves the homotopy splitting together imply that $\Psi^{L, \varepsilon_-}_1([\gamma^-_{ij}]) = [\gamma^+_{ij}]$ and $\Psi^{L, \varepsilon_-}_1([\delta^-_{ij}]) = [\delta^+_{ij}]$. The bounds in the theorem now follow from applying (\ref{eq:length-1-mono}) to the capacities computed above.
	
	The construction to show that the bounds in the theorem are the best possible is entirely similar to that in
	Theorem~\ref{thm:hopf}, where once again the restrictions on the auxiliary function $\rho(s)$ precisely match the capacity inequalities.    
	\end{proof}

\subsection{Non-Trivial Loops of Isotopies}
\label{ssec:loops}

Non-trivial loops of Legendrian isotopies yield another interesting set of Lagrangian cobordisms.  To clarify the meaning of a non-trivial loop of isotopies,  if $\mathcal{L}(M,\xi)$ is the space of Legendrian submanifolds of $(M,\xi)$, suitably topologized, then we  consider non-trivial elements of $\pi_1(\mathcal{L}(M,\xi), \leg)$,  especially those that are contractible in the space of smooth submanifolds.  The first such example, due to K\'alm\'an \cite{kalman:mono1}, is the loop of Legendrian trefoils in the standard contact $\rr^3$ pictured in Figure~\ref{fig:trefoil-loop}.  We will also consider a higher-dimensional example due to the first author and Sullivan \cite{ss:pi-k}.

\begin{figure} 
\begin{center}
\includegraphics[width=5in]{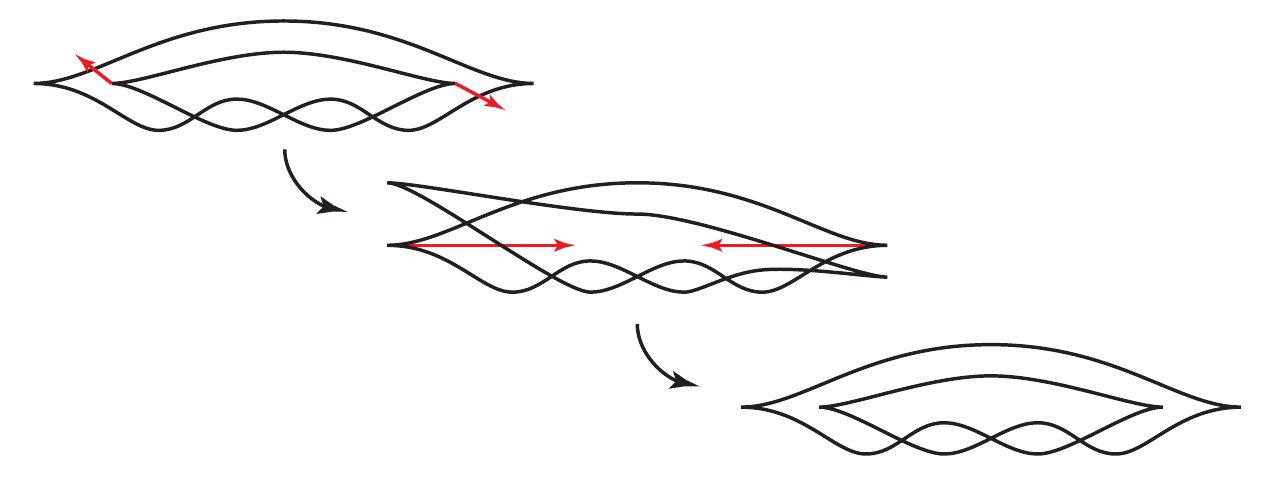} 
\caption{A non-trivial loop of Legendrian trefoils.}
\label{fig:trefoil-loop}
\end{center}
\end{figure}

First, consider the Lagrangian projection of the Legendrian trefoil shown in Figure~\ref{fig:trefoil-lagr}. It is straightforward to compute using the inequalities in \cite[Example 2.7]{kalman:mono1} that if the heights of the Reeb chords  satisfy the following constraints, then we obtain  a valid Lagrangian diagram:
	\begin{equation}
	\begin{split}	
		h(a_1) &> h(b_1)+h(b_2)+h(b_3), \\
		h(a_2) + h(b_1) + h(b_3) &> 2 h (a_1).
\end{split}
\end{equation}
Let us  denote $h(b_1) = h(b_3)$ by $h_1$ and $h(b_2)$ by  $h_2$.

\begin{prop}  \label{prop:kalman}   Let $L$ denote a Lagrangian cobordism induced by K\'alm\'an's non-trivial isotopy.  For any $L' \in [L]$, the length of $L'$ is bounded below by $|\ln h_1 - \ln h_2|$.
\end{prop}

\begin{proof} 
	A standard computation, as in \cite[Example 2.7]{kalman:mono1}, shows that the generators $a_1, a_2$ of the Legendrian contact homology algebra have degree $1$ and that the others have degree $0$.  Further, following \cite[Example 2.14]{kalman:mono1}, we may compute the differential to be:
	\begin{align*}
		\df a_1 &= 1+b_1+b_3+b_1b_2b_3 \\
		\df a_2 &= b_2+b_2b_3+b_1b_2+b_2b_3b_1b_2\\
		\df b_{1}&= \df b_{2} = \df b_{3} = 0.
	\end{align*}
	
	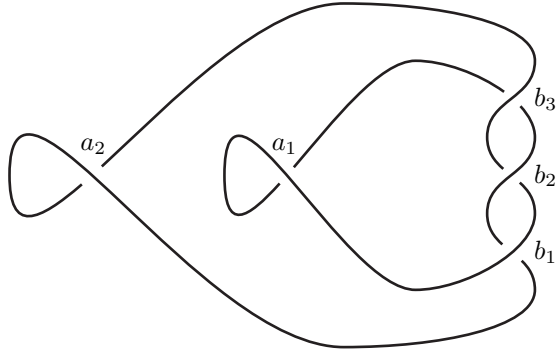
\begin{figure}
		\labellist
		\small
		\pinlabel $a_2$ [b] at 40 74
		\pinlabel $a_1$ [b] at 112 74
		\pinlabel $b_3$ [l] at 203 98
		\pinlabel $b_2$ [l] at 203 69
		\pinlabel $b_1$ [l] at 203 40 
		\endlabellist
	\centerline{\includegraphics{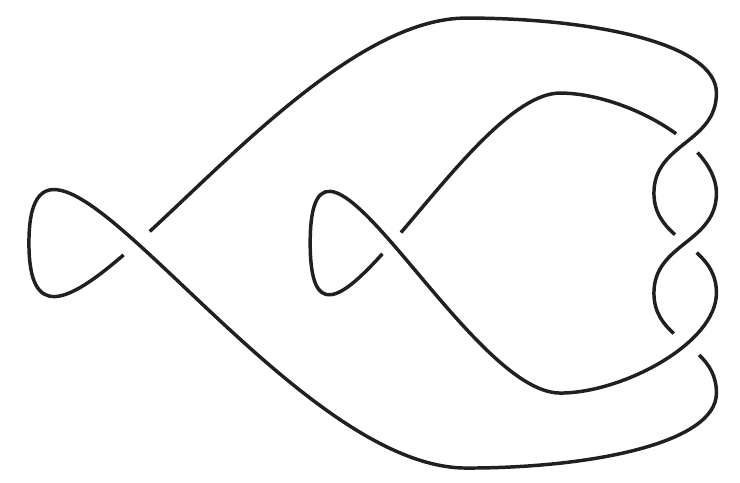}}
	\caption{The Lagrangian projection of the trefoil that K\'alm\'an uses as the base point for his non-trivial loop of Legendrians.}
	\label{fig:trefoil-lagr}
	\end{figure}

	Fixing the augmentations $\varepsilon_-$ that sends $b_1$ to $1$ and all other generators to $0$ and $\varepsilon_+$ that sends both $b_1$ and $b_2$ to $1$ and all other generators to $0$, we may compute that 
	\begin{align*}
		LCH^*(\leg, \varepsilon_-) &= \langle [a_2], [b_2], [b_1+b_3] \rangle \\
		LCH^*(\leg, \varepsilon_+) &= \langle [a_1] = [a_2], [b_2], [b_3] \rangle 	
	\end{align*}
Since we  computed linearized \emph{co}homology above,  we had to change coordinates, linearize the differential $\df^\varepsilon$, and then take adjoints to obtain $d^\varepsilon$.
		
	Using the fact that Chekanov's original description of the maps between Legendrian contact homology algebras induced by Reidemeister moves coincides with the chain-level cobordism map $\phi_L$ when $L$ is induced by the corresponding Reidemeister move \cite{ehk:leg-knot-lagr-cob}, we may say that in \cite[Example 3.6]{kalman:mono1}, K\'alm\'an computes the cobordism map for the Lagrangian induced by his loop of trefoils on the $b_*$ generators to be:
	\begin{align*}
		\phi_L(b_1) &= 1 + b_2 b_3 \\
		\phi_L(b_2) &= b_1 \\
		\phi_L(b_3) &= b_2
	\end{align*}
	We have made the obvious identification of the generators at the top and bottom of the cobordism.  We see immediately that $\varepsilon_+ = \varepsilon_- \circ \phi_L$.  
	
	Finally, we compute the linearized cobordism map on the $b_*$ generators to be:
	\begin{align*}
		\psi^{L, \varepsilon_-}_1(b_1) &= b_2 \\
		\psi^{L, \varepsilon_-}_1(b_2) &= b_3 \\
		\psi^{L, \varepsilon_-}_1(b_3) &= 0.
	\end{align*}
	On cohomology, we see that $\Psi^{L, \varepsilon_-}_1$ sends $[b_1+b_3]$ to $[b_2]$ and $[b_2]$ to $[b_3]$. Applying Theorem~\ref{thm:mono} to this cobordism map yields the inequalities 
	\begin{align*}
		\ell(L) &\geq \ln c(\leg, \varepsilon_-, [b_1+b_3]) -\ln c(\leg,\varepsilon_+, [b_2])) = \ln h_1 - \ln h_2, \\
		\ell(L) &\geq \ln c(\leg, \varepsilon_-, [b_2]) - \ln c(\leg,\varepsilon_+, [b_3]) = \ln h_2 - \ln h_1.
	\end{align*}
	We conclude that the length of the cobordism $L$ induced by K\'alm\'an's nontrivial isotopy is bounded below by $|\ln h_1 - \ln h_2|$.
\end{proof}

By choosing $h_1$ and $h_2$ appropriately, Proposition~\ref{prop:kalman} proves the second half of Theorem~\ref{thm:loops}:

\begin{cor}  For any $b > 0$, there exists a non-trivial loop $\gamma$ of Legendrian submanifolds with induced Lagrangian cobordism $L_\gamma$ such that the length of any Lagrangian cobordism in $[L_\gamma]$ is bounded below by $b$. 
\end{cor}

The first half of Theorem~\ref{thm:loops} follows from the following proposition:

\begin{prop}  There exists a non-trivial loop $\gamma$ of  Legendrian $n$-spheres in $J^1\rr^n$ with induced Lagrangian cobordism $L_\gamma$ such that    $[L_\gamma]$ 
contains arbitrarily short cobordisms. 
\end{prop}

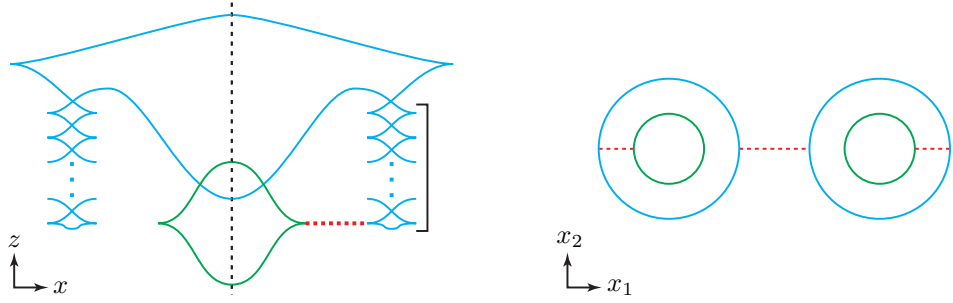
\begin{figure}
\labellist
\small
\pinlabel $x$ [l] at 23 6
\pinlabel $z$ [b] at 7 23
\pinlabel $x_1$ [l] at 307 6
\pinlabel $x_2$ [b] at 292 23 
\endlabellist
\centerline{\includegraphics[width=5in]{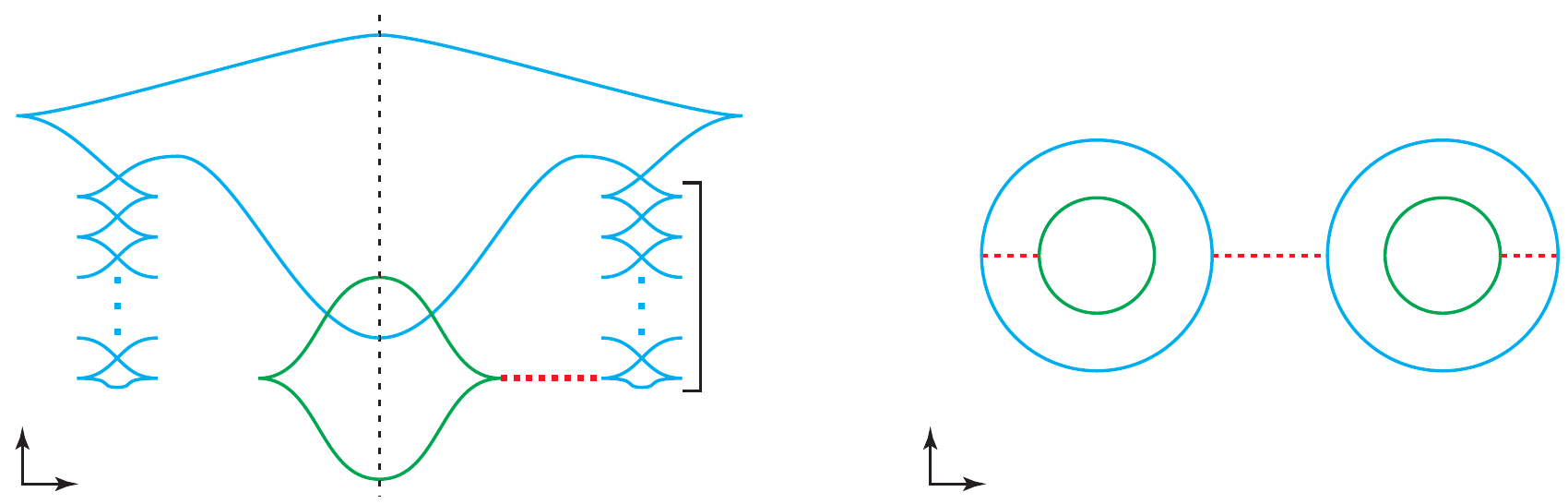} }
\caption{(a) A Legendrian sphere is obtained after spinning the diagram and then performing surgery along the dotted line. (b) Via an additional surgery, a connect sum of two copies
of this sphere
produces a sphere that has a front projection that is invariant under a  $180^{\circ}$ rotation of the $\mathbf x$-coordinates.
This rotation produces a non-trivial loop of Legendrian spheres \cite{ss:pi-k}.}
\label{fig:short-non-trivial-loop}
\end{figure}

\begin{proof} In \cite{ss:pi-k},
the first author and Sullivan constructed a non-contractible loop of Legendrian $n$-spheres in $J^1\rr^n$.   Briefly, the base point $\leg$ of the loop is constructed by first forming a sphere $\leg^r$ for $r \gg n$ using spinning and surgery; see Figure~\ref{fig:short-non-trivial-loop}(a) for a schematic picture.  Two copies of $\leg^r$, with one component shifted in the $x_1$ direction and rotated by $\pi$, are then joined together to form $\leg$ using a connect sum, as in Figure~\ref{fig:short-non-trivial-loop}(b).  The loop itself is simple to describe:  rotate the front projection of $\leg$ by $\pi$ in the $\mathbf{x}$ coordinates.  This rotation extends to a form-preserving diffeomorphism of $J^1 \rr^n$, so Corollary~\ref{cor:base-diffeo} implies that cobordisms induced by this isotopy may be arbitrarily short.
\end{proof}

\bibliographystyle{amsplain} 
\bibliography{main}

\end{document}